\documentclass[11pt]{article}

\usepackage[utf8]{inputenc}
\usepackage[english]{babel}
\usepackage[titletoc]{appendix}
\usepackage[centertags]{amsmath}
\usepackage{amsthm, amssymb}
\usepackage{dsfont}
\usepackage{bbm}
\usepackage{graphicx}
\usepackage{pgf,tikz}
\usepackage{caption}

\usepackage[pdftex,
            pdfauthor={M. Grieshammer},
            pdftitle={Genealogical distance under selection},
            pdfproducer={pdfLaTex},
            pdfcreator={}]{hyperref}

\usepackage{fancyhdr}
\newtheorem{theorem}{Theorem}[section]
\newtheorem{proposition}[theorem]{Proposition}

\newtheorem{lemma}[theorem]{Lemma}

\makeatletter                                           
\def\th@newremark{\th@remark\thm@headfont{\bfseries}}   
\makeatletter                                           

\theoremstyle{definition}

\theoremstyle{newremark}
\newtheorem{remarkx}[theorem]{Remark}

\newenvironment{remark}
  {\pushQED{\qed}\remarkx}
  {\popQED\endremarkx}

\newcommand{\eps}{\varepsilon}
\def\Simpl{\mathbb{S}}

\allowdisplaybreaks[4]

\begin{document}

 \title{Genealogical distance under selection}
 \author{Max Grieshammer\footnote{Institute for Mathematics, Friedrich-Alexander Universität Erlangen-Nürnberg, Germany;  	max.grieshammer@math.uni-erlangen.de, MG was supported by SPP 1590}
 }
 
 \maketitle
 
\thispagestyle{empty}

\begin{abstract}
We study the genealogical distance of two randomly chosen individuals in
a population that evolves according to a two type Moran model with
mutation and selection. We prove that this distance is stochastically
smaller than the
corresponding distance in the neutral model, when the population size is
large.
Moreover, we prove convergence of the genealogical distance under
selection to the distance in the neutral case, when the system is in equilibrium and
the selection parameter tends to infinity.

\end{abstract}

\\

\smallskip


\tableofcontents

\section{Introduction}

The study of genealogies in population models has a long history. 
In the case of neutral Wright-Fisher populations or Moran models Kingman \cite{kingman1982coalescent} introduced 
a coalescent process that can be used to analyze genealogical quantities.  
To be able to include selection in the genealogical model, Krone and Neuhauser \cite{Krone}, \cite{neuhauser1997genealogy}
introduced the so called ancestral selection graph. They also were the first studying the 
typical genealogical distance of two randomly chosen individuals under selection, which is also the object we are interested in. 
They gave a first order expansion of the Laplace transform of the distance in the selection parameter and showed that 
the linear term vanishes. 
There are other ideas on analyzing the genealogy in the Wright-Fisher population like, for example, the lookdown construction of Donelly and Kurtz \cite{donnelly1996countable}, \cite{donnelly1999genealogical}. 
But, as far as we know, there were no further results on the genealogical distance of two individuals when selection is present. 

We note that most of the above methods, on analyzing genealogical properties, have in common that they fix a time and construct or read off the genealogy by looking backward in time. 
In contrast to this approach  Depperschmidt, Greven and Pfaffelhuber \cite{DGP12} considered the genealogy as an evolving object and 
introduced a tree-valued Markov process that describes dynamically 
the genealogical relationship in the Wright-Fisher population with mutation and selection. Using generator calculations under equilibrium, 
this approach has the advantage  that one can find (relatively easy) recurrence relations for the genealogical 
distance of two individuals. As a result they were able to generalize the first order expansion of Krone and Neuhauser to a second 
order expansion and  proved that for small selection parameter the distance under selection is shorter (in the Laplace order) 
compared to the neutral model. 

We will also follow a forward in time approach to prove that the typical 
genealogical distance of two randomly chosen individuals is shorter under selection compared to neutrality. Our result generalizes the existing results
in two ways namely we have the dominance not only in the Laplace 
but also in the {\it stochastic} order and we are not restricted to small selection parameters nor an equilibrium 
population.  Instead we can show this result for {\it all} times and {\it all} selection parameters. \\

We consider a population of $N \in \mathbb N$ individuals, where each individual $i$ carries some type $u_i \in \{0,1\}$, and evolves according to the following dynamic: At rate $\binom{N}{2}$ we pick two individuals $i,j$ uniformly without replacement from the population and, with probability $\frac{1}{2}(1 + \frac{\alpha}{N}(u_i - u_j))$,  individual $j$ dies and is replaced by an offspring of  individual $i$ and, with probability $\frac{1}{2}(1 + \frac{\alpha}{N}(u_j - u_i))$, individual $i$  dies and is replaced by an offspring of  individual $j$. 
We call $\alpha \ge 0$ the \emph{selection} parameter and note that type $1$ has a selective advantage. In the case where $j$ is replaced by an offspring of $i$, we call $i$ the ancestor of $j$, and, in the case where $i$ is replaced by an offspring of $j$, we call $j$ the ancestor of $i$. Note that the child of an individual $i$ always has the same type as its parent.  \par 
In addition to the above transition, at rate $N\cdot \vartheta_0$ ($N \cdot \vartheta_1$) a randomly chosen unfit (fit) individual $i$ mutates to a fit (unfit) individual. 
We call $\vartheta_0,\vartheta_1 \ge 0$ the \emph{mutation} parameters (see Section \ref{sec.model} for the definition of the model). \\

We note that the above construction generates a distance $r_t$ on the set of individuals, called the \emph{genealogical distance}, where 
$r_t(i,j)$ is defined as the time to the most recent common ancestor of two individuals $i,j$ in the today's (i.e. at time $t$ measured from the beginning of time) population (see Section \ref{sec.construction}). In this paper we are interested in 
this distance, to be precise we are interested in the random variable $R^{N,\alpha}_t$ given by
\begin{equation}
P(R^{N,\alpha}_t \le h) := \frac{1}{N^2}\sum_{i,j = 1}^N P(r_t(i,j) \le h),\qquad h\ge 0.
\end{equation}
Note that 
\begin{equation}
P(R^{N,\alpha}_t \le h) =   E\left[\mu^N\otimes \mu^N(\{i,j: r_t(i,j) \le h\})\right],
\end{equation}
where $\mu^N$ is the uniform distribution on the set of individuals $\{1,\ldots,N\}$. In this sense,  
\begin{itemize}
\item[] {\it $R^{N,\alpha}$ is the typical genealogical distance of two randomly chosen individuals.}
\end{itemize}

To analyze this quantity it is convenient to consider the large population limit and one can prove (see \cite{DGP12} - compare also Section \ref{sec.tree-valued}) that $R^{N,\alpha}:= (R^{N,\alpha}_t)_{t \ge 0} \Rightarrow R^\alpha$. \par 
At this point we note that  the frequency of type $1$  process $\bar Y^N:=\bar Y^{N,\alpha}:= (\bar Y_t^{N,\alpha})_{t \ge 0}$ converges for $N \rightarrow \infty$,  where the limit process, $\bar Y$, is called the  \emph{Wright-Fisher diffusion with mutation and selection} (see for example Section 10.2 in \cite{EK86}). This process can be characterized via the following operator $G$, that acts on twice continuously differentiable functions, $C^2([0,1])$:
\begin{equation}
Gf(x) := \frac{1}{2} x(1-x) \frac{\partial^2}{\partial x^2} f(x) + (-\vartheta_1 x + \vartheta_0 (1-x) + \alpha x (1-x))  \frac{\partial}{\partial x} f(x).
\end{equation}
In this sense we call $R^\alpha$ the genealogical distance in a Wright-Fisher population.  \\

In this paper we will prove that there is a coupling such that
\begin{equation}
P(R^\alpha_t \le R^0_t) = 1,\qquad \text{for all } \alpha \ge 0 \text{ and all } t \ge 0,
\end{equation}
i.e. the genealogical distance in the Wright-Fisher population under selection at all times $t$ is stochastically smaller than 
the distance under neutrality, where we assume that at the beginning of time, $t = 0$, all individuals are related, i.e. $R^\alpha_0 = R^0_0 = 0$ (see Section \ref{sec.main.res}). Note that in this case, the distance under neutrality is given by 
\begin{equation}
P(R^0_t \le h) = \left\{ \begin{array}{ll}
1-e^{-h},&\quad h < t\\
1,&\quad h \ge t.
\end{array}\right.
\end{equation}
That is, the law of the distance under neutrality is (up to time $t$) the exponential distribution. \par

The result gives the (much more general) answer to a question stated in \cite{DGP12}, namely if for small $\alpha$ one has $R^\alpha \le R^0$ stochastically. In their work they proved that when $\alpha$ is small and the system is in equilibrium, one has domination in the Laplace order, which is,  in general, weaker than the stochastic order. \\ 

It is known (see again \cite{DGP12}) that for small selection parameters (i.e. $\alpha \rightarrow 0$) one has $R^\alpha \Rightarrow R^0$
and it is (as far as we know) still open to prove what happens when the selection parameter is large (i.e. $\alpha \rightarrow \infty$). 
\cite{DGP12} suggested that in this case there are only fit types left in the population and hence there 
is no selective advantage and therefore the genealogical distance should look similar to the neutral case. We give a proof of this claim, 
namely we prove that when the system is in equilibrium, one has (see Section \ref{sec.main.res})
\begin{equation}
R^\alpha_\infty \Rightarrow R^0_\infty,\qquad \text{for }  \alpha \rightarrow \infty.
\end{equation}

\section{The genealogical distance in the Moran model}

In this section we give a formal definition of the quantities we are interested in. In Section \ref{sec.model}
we define the model. In Section \ref{sec.construction} we give the construction of the model in terms of Poisson 
processes, which allows us to define the genealogical distance. Finally, in Section \ref{sec.tree-valued}, we give a 
very short introduction to the tree-valued model, i.e. the model with values in marked-(ultra-)metric measure spaces, and 
sketch how this setup can be used to prove convergence of the genealogical distance. 

\subsection{The model}\label{sec.model}

We want to describe the genealogy of a population, consisting of  $N \in \mathbb N$ individuals, where we assume that all individuals $k \in \{1,\ldots,N\}$ carries some type $u_k \in \mathbb K := \{0,1\}$. We assume further, that the types $(u_k(t))_{k \in \{1,\ldots,N\}}$ evolve according to the following dynamics:

\begin{itemize}
\item[(1)] {\it Resampling:} Every pair $i \neq j$ is replaced with rate one. 
If such an event occurs, $i$ is replaced by an offspring of $j$ with probability $\frac{1}{2}$, or $j$ is replaced by an
offspring of $i$ with probability $\frac{1}{2}$ and the offspring always carries the type of its parent. 
\item[(2)] {\it Mutation:} The type of an individual $i$ changes (independent of the others) from $u_i$ to $1-u_i$ with the {\it mutation rates} 
\begin{equation}
\begin{split}
\vartheta_0 \ge 0,& \quad \text{if } u_i = 0, \\
\vartheta_1 \ge 0,& \quad \text{if } u_i = 1. 
\end{split}
\end{equation}
\item[(3)] {\it Selection:}  For a pair $i \neq j$ at rate
\begin{equation}
\frac{\alpha}{N} \cdot u_i
\end{equation}
a selection event takes place, i.e. individual $j$ is replaced by an offspring of individual $i$, and we call $\alpha \ge 0$ the {\it selection parameter}. Note that type $1$ has a selective advantage compared to $0$ and we call $1$ the {\it fit} type. 
\end{itemize}

\begin{remark}
In order to analyze this model, we will consider the large population limit (i.e. $N \rightarrow \infty$). In this situation we note that 
even though the above model differs a bit from the model described in the introduction, the large population limits coincide. And we decided to use the above construction for technical reasons. 
\end{remark}

\subsection{Graphical construction and the genealogical distance} \label{sec.construction}

Here we give the formal (graphical) construction of the above model, where we follow the approach in \cite{gpw_mp} and \cite{DGP12}. Let $I_N:= \{1,\ldots,N\},\ N \in \mathbb N$, $\mathbb K := \{0,1\}$ and 
\begin{equation}\label{ppp}
\left\{\eta^{i,j}_{\text{res}}:\ i,j \in I_N,\ i \not=j\right\}
\end{equation}
be a realization of a family of independent rate $1$ Poisson point processes,
\begin{equation}\label{ppp_sel}
\left\{\eta^{i,j}_{\text{sel}}:\ i,j \in I_N,\ i \not=j\right\}
\end{equation}
be a realization of a family of independent rate $\frac{\alpha}{N} $ Poisson point processes
and 
\begin{equation}\label{ppp_mut}
\left\{\eta^{i}_{\text{mut}_0}:\ i \in I_N\right\},\qquad \left\{\eta^{i}_{\text{mut}_1}:\ i \in I_N\right\}
\end{equation}
be a realization of families of independent rate $\vartheta_0 $, $\vartheta_1$ Poisson point processes. We assume
that all random mechanisms are independent of each other.\\

Now we consider the type process $(u_i(t))_{i \in I_N, t \ge 0}$, that starts in some initial value $(u_1(0),\ldots,u_N(0)) \in \mathbb K^{I_N}$, where we assume for the rest of this paper: 
\begin{itemize}
\item[] {\it $(u_1(0),\ldots,u_N(0)) = (\bar u_1,\ldots,\bar u_N)$, where $(\bar u_i)_{i \in \mathbb N}$ are independent and identically distributed }
\end{itemize}
 and evolves according to the following rules:
If $\eta^{i,j}_{\text{res}}(\{t\}) = 1$, then we set $u_j(t) = u_i(t-)$ (such an event is called resampling event). If $\eta^{i,j}_{\text{sel}}(\{t\}) = 1$, then we set $u_j(t) = u_i(t-)$ if $u_i(t-) = 1$ 
(such an event is called selection event). If $\eta^{i}_{\text{mut}_0}(\{t\}) = 1$, then we set $u_i(t) = 1$ if $u_i(t-) = 0$ 
and if $\eta^{i}_{\text{mut}_1}(\{t\}) = 1$, then we set $u_i(t) = 0$ if $u_i(t-) = 1$ (such events are called mutation events).\\

Next we need the notion of ancestors of two individuals $i$ and $j$ at time $t$ and we start with the following definition: 

\begin{itemize}
\item[] For $i,i' \in I_N$, $0\le h< t < \infty$ we say that there is a {\it path} from $(i,h)$ to $(i',t)$ if there is an  $n \in \mathbb N$, 
$h \le t_1 < t_2 < \cdots < t_n \le t$ and
$j_1,\ldots,j_n \in I_N$ such that for all $k \in \{1,\ldots,n+1\}$ ($j_0:= i, j_{n+1}:=i'$)
$\eta^{j_{k-1},j_k}_{\text{res}}\{t_k\} = 1$ or $\eta^{j_{k-1},j_k}_{\text{sel}}\{t_k\} \cdot u_{j_{k-1}}(t_k-) = 1$, $\eta^{x,j_{k-1}}_{\text{res}}((t_{k-1},t_k))= 0$ and $\eta^{x,j_{k-1}}_{\text{sel}}(\{s\}) \cdot u_{x}(s-)= 0$ for all $s  \in (t_{k-1},t_k)$ and $x \in I_N$.
\end{itemize}

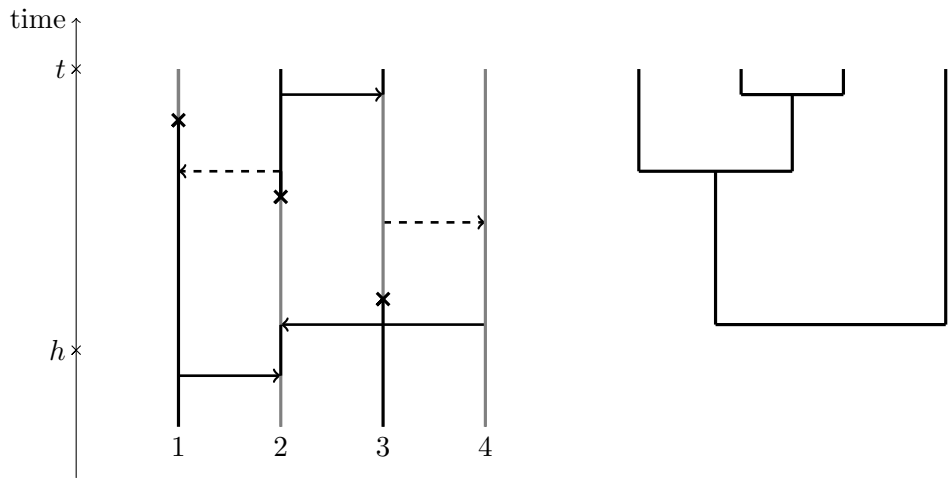
\begin{figure}[ht]
\begin{center}
\begin{tikzpicture}[scale = 0.68]
\draw [->] (2,-4) -- (2,5);
\draw [line width=1.2pt] (4,-3)-- (4,4);
\draw (6,-3)-- (6,4);
\draw (8,-3)-- (8,4);
\draw (10,-3)-- (10,4);
\draw [->,line width=1.pt] (4,-2) -- (6,-2);
\draw [->,line width=1.pt] (10,-1) -- (6,-1);
\draw [->,line width=1.pt,dash pattern=on 3pt off 3pt] (8,1) -- (10,1);
\draw [->,line width=1.pt,dash pattern=on 3pt off 3pt] (6,2) -- (4,2);
\draw [->,line width=1.pt] (6,3.5) -- (8,3.5);
\draw [line width=1.2pt,color=gray] (4,3)-- (4,4);
\draw [line width=1.2pt,color=gray] (6,-3)-- (6,-2);
\draw [line width=1.2pt] (6,-2)-- (6,-1);
\draw [line width=1.2pt,color=gray] (6,-1)-- (6,1.5);
\draw [line width=1.2pt] (6,1.5)-- (6,2);
\draw (6,1.5)-- (6,2);
\draw [line width=1.2pt] (6,1.5)-- (6,4);
\draw [line width=1.2pt,color=gray] (10,-3)-- (10,4);
\draw [line width=1.2pt] (8,-3)-- (8,-0.5);
\draw [line width=1.2pt,color=gray] (8,-0.5)-- (8,3.5);
\draw [line width=1.2pt] (8,3.5)-- (8,4);
\draw (4,-3) node[anchor=north] {$1$};
\draw (6,-3) node[anchor=north] {$2$};
\draw (8,-3) node[anchor=north] {$3$};
\draw (10,-3) node[anchor=north] {$4$};
\draw (2,4) node[anchor=east] {$t$};
\draw (2,-1.5) node[anchor=east] {$h$};
\draw (2,5) node[anchor=east] {$\text{time}$};
\draw [line width=1.2pt] (15,4)-- (15,3.5);
\draw [line width=1.2pt] (15,3.5)-- (17,3.5);
\draw [line width=1.2pt] (17,3.5)-- (17,4);
\draw [line width=1.2pt] (13,4)-- (13,2);
\draw [line width=1.2pt] (16,3.5)-- (16,2);
\draw [line width=1.2pt] (13,2)-- (16,2);
\draw [line width=1.2pt] (19,4)-- (19,-1);
\draw [line width=1.2pt] (19,-1)-- (14.5,-1);
\draw [line width=1.2pt] (14.5,-1)-- (14.5,2);
\begin{scriptsize}
\draw [line width=1.2pt] (4,3)-- ++(-3.5pt,-3.5pt) -- ++(7pt,7pt) ++(-7pt,0) -- ++(7pt,-7pt);
\draw [line width=1.2pt] (6,1.5)-- ++(-3.5pt,-3.5pt) -- ++(7pt,7pt) ++(-7pt,0) -- ++(7pt,-7pt);
\draw [line width=1.2pt] (8,-0.5)-- ++(-3.5pt,-3.5pt) -- ++(7pt,7pt) ++(-7pt,0) -- ++(7pt,-7pt);
\draw (2,4)-- ++(-2.5pt,-2.5pt) -- ++(5pt,5pt) ++(-5pt,0) -- ++(5pt,-5pt);
\draw (2,-1.5)-- ++(-2.5pt,-2.5pt) -- ++(5pt,5pt) ++(-5pt,0) -- ++(5pt,-5pt);
\end{scriptsize}
\end{tikzpicture}
\caption{\label{fig.TVMMMS} \footnotesize On the left side we see the graphical construction of the Moran model with mutation and selection: $\rightarrow$ is a resampling arrow, $\dashrightarrow$ is a selective arrow and $\text{\bf X}$ indicates a mutation event.  $\bullet$ is the fit type, which corresponds to $1$, $\color{gray}{\bullet}$ is the unfit type, which corresponds to $0$. Note that selection arrows can only be used by a fit type. On the right side we see the genealogical tree of the population at time $t$. 
In this case the ancestor of all individuals at time $h$ would be individual $4$, i.e.  $A_h(i,t) = 4$ for all $i = 1,\ldots,4$}
\end{center}
\end{figure}

\noindent Note that for all $i \in I_N$ and $0 \le h \le t$ there exists an unique element
\begin{equation}
A_h(i,t) \in I_N
\end{equation}
with the property that there is a path from $(A_h(i,t),h)$ to $(i,t)$. We call $A_h(i,t)$ the {\it ancestor of  $(i,t)$ at time $h$} (see Figure \ref{fig.TVMMMS}). \par

Let $r_0$ be a (pseudo-ultra)metric on $I_N$ and  $i,j \in I_N$. Then we define the following (pseudo-ultra)metric on $I_N$:
{\small
\begin{equation}\label{eq.ultrametric}
r_t(i,j):=\left\{ \begin{array}{ll}
t-\sup\{h \in [0,t]:A_h(i,t)=A_h(j,t)\},&\quad \textrm{if } A_0(i,t) = A_0(j,t),\\[0.2cm]
t + r_0(A_0(i,t), A_0(j,t)) ,&\quad \textrm{if } A_0(i,t) \not= A_0(j,t).
\end{array}\right.
\end{equation}
}

We will in the following always assume that 
\begin{itemize}
\item[] {\it $r_0 \equiv 0$, i.e. all individuals are related at time $0$.}
\end{itemize}

Finally we define the random variable $R^{N,\alpha}_t$  by

\begin{equation}
P(R^{N,\alpha}_t\le h) := \frac{1}{N^2} \sum_{i,j = 1}^N P(r_t(i,j) \le h).
\end{equation} 

In the following we will abbreviate $R^{N}_t := R^{N,\alpha}_t$, when the context is clear. 

\begin{remark}
Note that by the above assumption, we have $P(R^N_t \le h) = 1$ for all $h \ge t$. 
\end{remark}

\subsection{A note on the tree-valued Moran model and the genealogical distance in the large population limit}\label{sec.tree-valued}

Since we need some of the results for the tree-valued Moran model, we give here a (very) short introduction to this model. Define  $\mu^N\in \mathcal M_1(I_N\times \mathbb K)$ by
\begin{equation}
\mu^N_t= \frac{1}{N} \sum_{k \in I_N} \delta_{(k,u_k(t))}. 
\end{equation}

Since $r_t$, defined in the previous section, is only a pseudo-metric, we  consider the following equivalence relation $\approx_t$  on $I_N$: 
 $x \approx_t y \Leftrightarrow r_t(x,y) = 0$.
We denote by $\tilde I_N^t:= I_N\! /\!\!\approx_t$ the set of equivalence classes and note that we can find a set of representatives $\bar I_N^t$
such that $\bar I_N^t \rightarrow \tilde I_N^t,\  x \to [x]_{\approx_t}$ is a bijection. We define 
\begin{align}
\bar r_t(\bar  i,\bar  j) =  r_t(\bar i,\bar j),\quad \bar \mu^N_t(\{\bar i\}\times \cdot) &=\mu^N_t( [\bar i]_{\approx_t} \times \cdot),\qquad  \bar i,\bar j \in \bar I_N^t.
\end{align}

Then the {\it tree-valued Moran model with mutation and selection}, of size $N$ is defined as
\begin{equation}
\mathcal U_t^{N} := [\bar I_N^t,\bar r_t,\bar \mu^N_t], 
\end{equation}
where $[\cdot]$ denotes the equivalent class induced by the equivalence relation $\sim$, where roughly speaking two (marked) metric measure spaces are equivalent, if there is a measure preserving isometry $\varphi$ between the two metric measure spaces which also preserves the types (also called marks). 
For all interested readers, we refer to \cite{gpw_mp} (for the non marked setup) or \cite{DGP12} (for the tree-valued Moran model with mutation and selection and its large population limit; the so called tree-valued Fleming-Viot process). For all others, we summarize the results needed in this paper. 

\begin{proposition}\label{prop.gen.distance}
Assume that $(u_1(0),\ldots,u_N(0))$ are as above. Then 
\begin{equation}
\mathcal L(R^N_t) \Rightarrow \mathcal L(R_t),
\end{equation}
for all $t \ge 0$, where $R_t$ is a variable on $[0,\infty)$. Moreover, 
\begin{equation}
P(R^N_t \le x) \rightarrow P(R_t \le x), 
\end{equation} 
for all $x < t$ and $P(R^N_t \le x)  = P(R_t \le x) = 1$ for $x \ge t$. Finally, one has 
\begin{equation}
R_t \Rightarrow R_\infty,
\end{equation}
for $t \rightarrow \infty$ and all $\alpha \ge 0$ and $R_\infty$ does not depend on the initial type configuration. 
\end{proposition} 

\begin{remark}\label{r.WF.equi}
Note that the Wright-Fisher diffusion, described in the introduction has an unique stationary distribution which is given by the following 
density
\begin{equation}
\rho(x) = \frac{x^{2\vartheta_0 - 1} (1-x)^{2\vartheta_1 - 1} \exp(2 \alpha x)}{\int_0^1 x^{2\vartheta_0 - 1} (1-x)^{2\vartheta_1 - 1} \exp(2 \alpha x) dx}
\end{equation}
with respect to the Lebesgue measure (see Section 10.2 in \cite{EK86}). Applying Theorem 4 in \cite{DGP12} we can therefore interpret $R_\infty$ as the genealogical distance of this equilibrium Wright-Fisher population.
\end{remark}

\begin{proof}
The first part is Theorem 3 in \cite{DGP12} combined with the fact that $P(R^N_t \in \cdot) = E[\nu^{2,\mathcal U_t^{N}}(\cdot)]$, where 
\begin{equation}\label{eq.pairwise}
\begin{split}
\nu^{2,[X,r,\mu]}([0,h]) &:= \nu^{[X,r,\mu]}([0,h] \times \mathbb R^{\binom{\mathbb N}{2}} \times \mathbb K^{\mathbb N}) \\
&=\mu \otimes \mu(\{(x,u),(y,v):\ r(x,y) \le h\}),
\end{split}
\end{equation} 
is the distance matrix distribution (see Definition 3.4 and Definition 3.6 in \cite{DGP12} and recall that the 
map $\mu \mapsto \int f d\mu$ is continuous for all bounded continuous $f$). \par
The second part follows, since $x \mapsto E[\nu^{2,\mathcal U_t}([0,x])]$ is continuous (where $\mathcal U_t$ denotes the tree-valued Fleming-Viot process with mutation and selection). This follows by the fact, that under neutrality $E[\nu^{2,\mathcal U_t}([0,x])]$
is exponentially distributed (see Remark 3.16 in \cite{DGP12}) combined with the Girsanov-transform (see Theorem 2 in \cite{DGP12}) and the continuous mapping Theorem (see Theorem 8.4.1 in \cite{Bogachev}, Vol. II).
The last part is Theorem 4 in \cite{DGP12} combined with the above observation and Remark \ref{r.WF.equi}.
\end{proof}

\section{Results}

Here we give the main result of this paper and some further results and discussions that are related. 

\subsection{The main result}\label{sec.main.res}

Recall Proposition \ref{prop.gen.distance} and write $R^\alpha$ in order to indicate the dependence on the selection parameter $\alpha > 0$.  

\begin{theorem}\label{thm.main}
There is a coupling such that
\begin{equation}
P(R^\alpha_t \le R^0_t) = 1, \qquad \text{for all } t \ge 0 \text{ and } \alpha \ge 0. 
\end{equation}
Moreover, in the case where $\vartheta_1 = \vartheta_2 =:\vartheta$ one has 
\begin{equation}
\mathcal L(R^\alpha_\infty) \Rightarrow \mathcal L(R^0_\infty),\qquad \text{for } \alpha \rightarrow \infty.
\end{equation}
\end{theorem}

\begin{remark}
The restriction in the second result seems absolutely unnecessary. 
But, in the general case, it is much harder to show convergence of the moments of the one dimensional Wright-Fisher diffusion in equilibrium (compare Remark \ref{r.gen.mut}). 
\end{remark}

\subsection{Further observations}

In this section we remark some further results, that can be obtained using the methods presented in the proof of the main theorem. \\

\noindent {\bf (1)} Even though the result is given for the two type model, the method used for the proof can be generalized to an arbitrary number of types
by specifying suitable selection and mutation operators and we claim that 
\begin{itemize}
\item[] {\it The result in Theorem \ref{thm.main} stays valid in the model with  any finite number of types. }
\end{itemize}

\noindent {\bf (2)} The following result generalizes the observation in \cite{DGP12}.

\begin{proposition}\label{t.strong}
 Let $\vartheta_0 = \vartheta_1 =  \frac{1}{2}$ and  $h \ge 0$. Then
\begin{align}
P(R^\alpha_\infty\le h) = P(R^0_\infty\le h)+ \alpha^2 \cdot \tau(h) + o(\alpha^2), \quad \text{for} \quad \alpha \rightarrow 0,
\end{align}
with 
\begin{equation}
\tau(h) = -\frac{1}{735}e^{-8h} +\frac{1}{42}e^{-h}h +\frac{1}{60} e^{-3h}-\frac{3}{196} e^{-h}.
\end{equation}
\end{proposition}

\begin{remark}
(a) We note that by differentiation, one gets the following asymptotic density for $\alpha \rightarrow 0$: 

\begin{equation}
\begin{split}
P(R^\alpha_\infty \in dh) = &\left(e^{-h} + \left(\frac{8}{735}e^{-8h} -\frac{1}{42}e^{-h}h+ \frac{1}{42}e^{-h} \right.\right.\\
&{}\hspace{1cm}\left.\left.  -\frac{1}{20} e^{-3h}+\frac{3}{196} e^{-h} \right)\alpha^2 + o(\alpha^2)\right) 1(h \ge 0) dh.
\end{split}
\end{equation}

If we now calculate the Laplace transform: 
\begin{equation}
E[e^{-2 \lambda R^\alpha_\infty}] = \frac{1}{1+2\lambda} + \frac{1}{3}\cdot\frac{\lambda}{(4+\lambda)(3+2\lambda)(1+2\lambda)^2} \alpha^2 + o(\alpha^2),
\end{equation}
then this is exactly the Laplace transform given in Theorem 5 of \cite{DGP12} (where one has to choose $\gamma = 1$, $\vartheta_{\textcolor{gray}{\bullet}} = \vartheta_{\bullet} = 1$).\par 
(b) Since we defined the genealogical distance to be the time to the most recent common ancestor, while in \cite{DGP12} they defined the 
distance to be two times the distance to the most recent common ancestor, we need to multiply with two in the above transform. \par
(c) It is not necessary to choose   $\vartheta_0 = \vartheta_1 = \frac{1}{2}$, but the calculations for general $\vartheta_0,\vartheta_1$ are a bit harder. 
\end{remark}

\begin{center}
\begin{figure}[ht]
\includegraphics[scale =0.42]{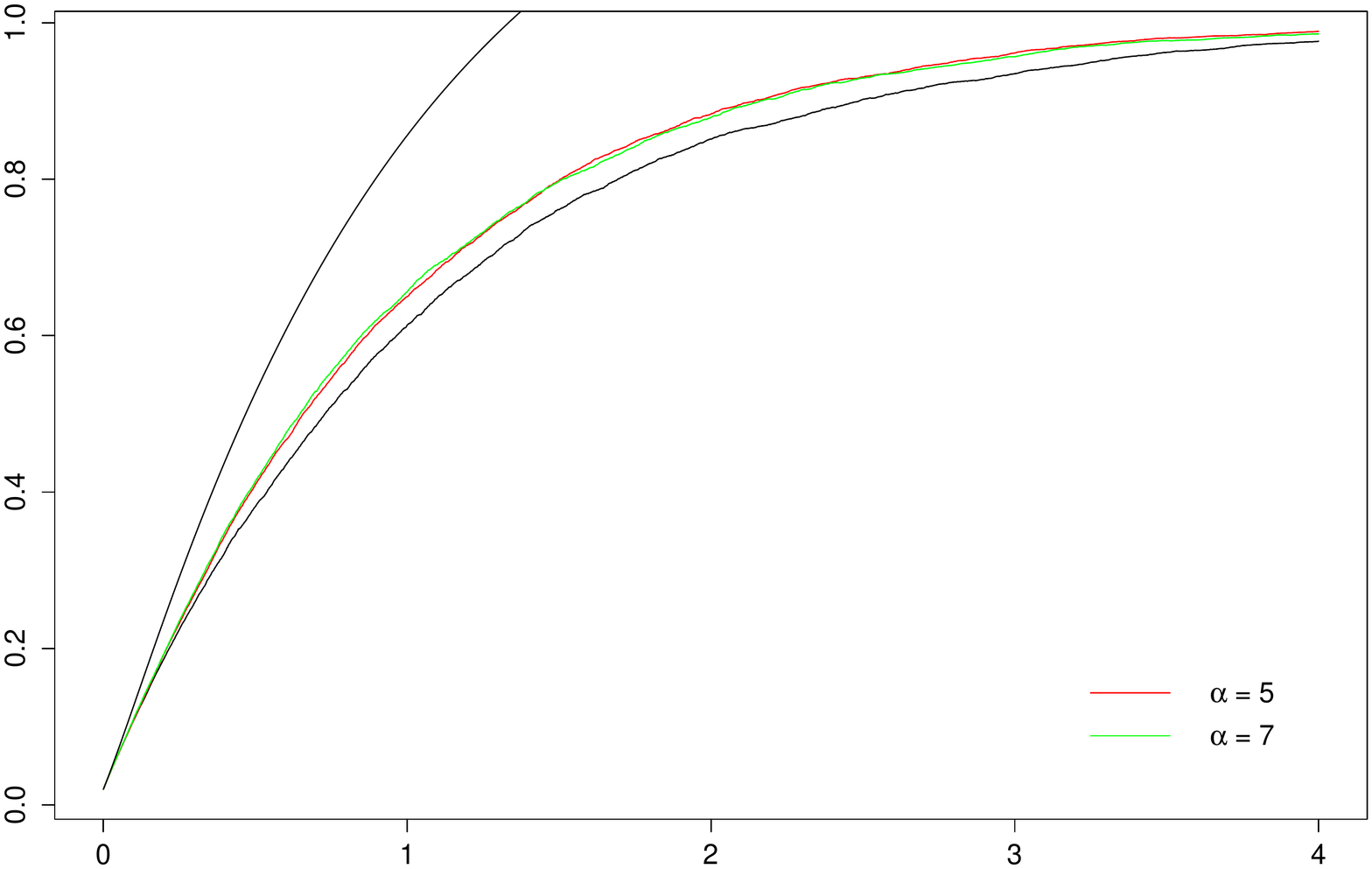} 
\vspace{-1cm} 
\caption{\footnotesize Comparison of the cumulative distribution functions (CDF) of the genealogical distance for different selection parameters, where $\vartheta_0 = \vartheta_1 = 1/2$; we assumed that the type process is in equilibrium. The lower black curve is the CDF of the distance under neutrality, i.e. the exponential distribution. The upper black curve is an upper bound for the genealogical distance given in the proof section.  The plot was created with {\normalfont{R}}.}
\label{fig:Comparison}
\end{figure}
\end{center}

\noindent {\bf (3)} One may ask if there is some sort of monotonicity in the parameter $\alpha$. Figure \ref{fig:Comparison} shows that we can not expect a monotonicity on the level of cumulative distributions functions (and therefore on the level of stochastic dominance) in general. 
The heuristic reason why this can not be true is the following: Assume initially there is a fixed (the same for all selection parameters)
positive fraction, say $\eps > 0$, of unfit types in the population. Then, when $\alpha \rightarrow \infty$, the unfit types will 
immediately be replaced by fit types, which implies that there is a positive fraction of individuals with genealogical distance $0$. Moreover, 
when there are only fit types left in the population, then there is no selective advantage and therefore the genealogical distance 
will look like the neutral distance. That means that $\mathcal L(R) \approx \eps \delta_0 + (1-\eps) e^{-t}dt$. But, since the initial number of unfit types tends to zero, when $\alpha \rightarrow \infty$, we are in the situation of Theorem \ref{thm.main}. \par
In other words, we have two effects: The larger $\alpha$ is the quicker the population can reduce distance of a small fraction of individuals with the price that 
after small  time the distance looks like the exponential distribution. On the other hand when $\alpha$ is small 
the system needs more time to reduce the distances but has a larger fraction of individuals where this reduction is possible. 
Hence, the two graphs might intersect each other, which is exactly what we see in Figure  \ref{fig:Comparison} (the two curves intersect at time 1.4). 

\begin{figure}[ht!]
\begin{center}
\includegraphics[scale = 0.42]{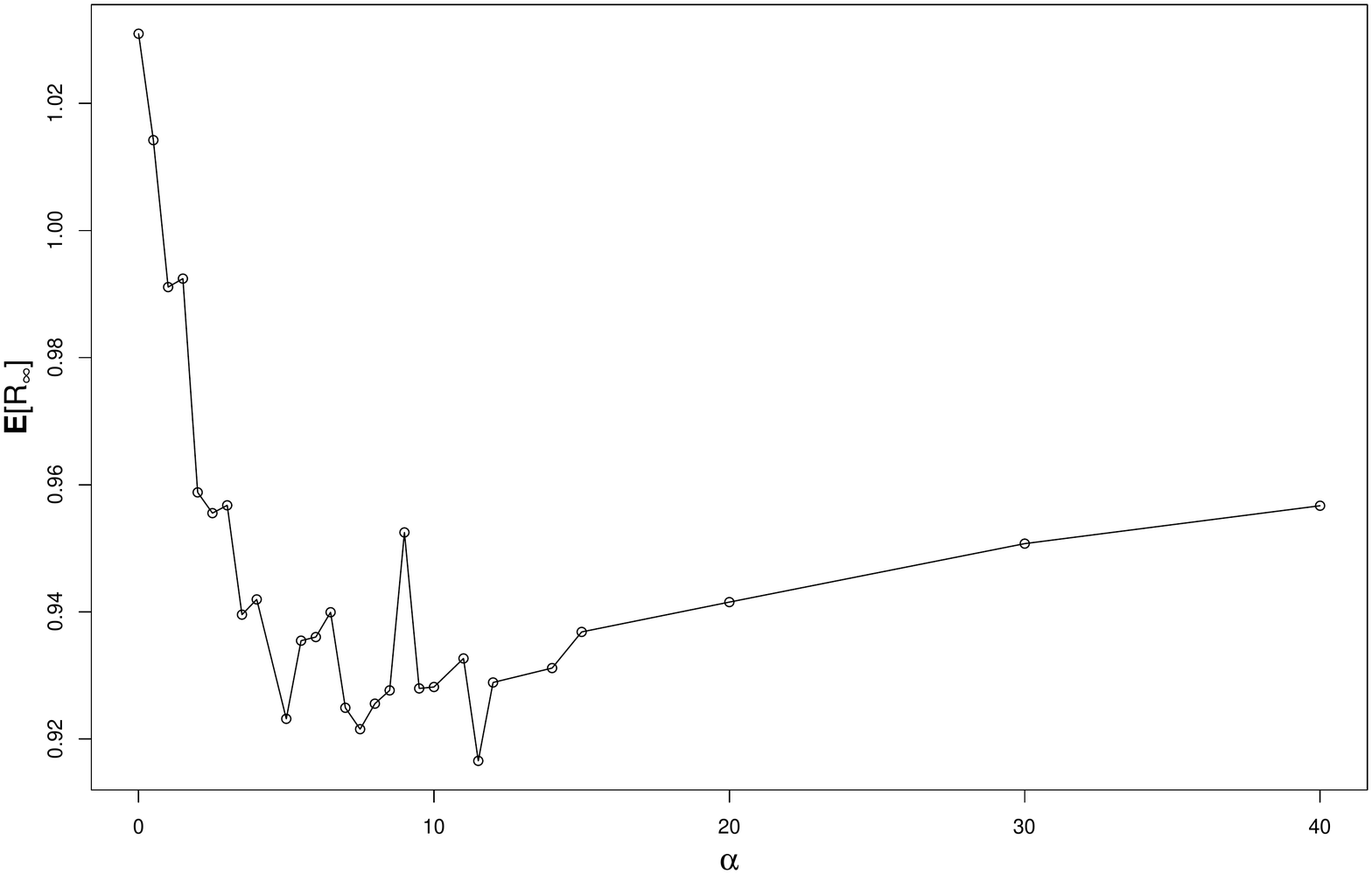} 
\end{center}
\vspace{-1cm} 
\caption{\footnotesize  The expected value of $R$ for different selection parameters.}
\label{fig:Expectation}
\end{figure}

Figure \ref{fig:Expectation} shows that we can not even expect monotonicity for the expectations.

\begin{remark}
As one can see in Figure \ref{fig:Expectation}, our simulations underestimate the real value on the level of CDFs and therefore overestimate the real value of the expectation. Nevertheless these plots give a hint, that things are more complicated than one would 
expect. 
\end{remark}

\section{The key observation}

Here we present the basic idea and tool for the poof of our result. Roughly speaking, instead of considering a backward in time dynamic,
which is in some sense the classical approach, we will use a forward in time representation, which is closely connected to the idea 
of evolving genealogies in the sense of tree-valued processes. This new approach has the advantage that we will get an approximation 
of the genealogical distance by solutions of an SDE. To be more precise, we can approximate the genealogical distance  by the sum 
of the squares of the coordinates of an $n$-dimensional Wright-Fisher diffusion with a certain selection and mutation operator and a certain initial condition, when we let $n $ go to infinity. \par 
In Section \ref{sec.idea} we present the basic idea, which we turn into a theorem in Section \ref{sec.formal}.

\subsection{The idea}\label{sec.idea}

Recall that 
\begin{equation}
P(R^N_t\le h) := \frac{1}{N^2} \sum_{i,j = 1}^N P(r_t(i,j) \le h)
\end{equation} 
and that $r_t(i,j) \le h$ if and only if, individual $i$ and $j$ in the time $t$ population had a common ancestor before 
time $h$ (measured backwards). Or, in other words, in the time $t-h$ population there was an individual that gave birth 
to both $i$ and $j$. Hence,  two individuals in a population at time $t$ are related at some time $h$ (measured backwards) if they are in the same family spanned by some common ancestor at time $t-h$. If we now ask for the probability that two randomly chosen individuals 
are related at time $h$ (measured backwards), then this is same as to ask for the probability that we sample two individuals from the same family. Therefore, 
if we want to calculate this probability, we only need to know how big the different families are. To do so, we start a process at time $t-h$ 
with  $N$ individuals and each individual gets a label (or "type") $\kappa_1,\ldots,\kappa_N$, where we interpret the mass of type $\kappa_i$ as the family size spanned by the individual $i$. That is, for example, at time $0$ (which corresponds to the time $t-h$) all individuals are in their 
own family, i.e. the relative frequency, which we denote by $\mathfrak{f}_1,\ldots,\mathfrak{f}_N$, of the family sizes is $1/N,\ldots,1/N$. Now, when time evolves, a family, say with label $\kappa_i$, 
gets a new member due to resampling, say one member of the family labeled by $\kappa_j$ (or in other words one member of the family $\kappa_j$ dies and gets replaced by a member of $\kappa_i$), i.e.  $\mathfrak{f}_i \to \mathfrak{f}_i +1/N$, $\mathfrak{f}_j \to \mathfrak{f}_j - 1/N$ (see Figure \ref{fig.1}).

\begin{figure}[ht!]
\begin{center}
\begin{tikzpicture}[scale = 0.55]
\draw [line width=0.8pt] (1,4)-- (1,-4);
\draw [line width=0.8pt] (3,4)-- (3,-4);
\draw [line width=0.8pt] (5,4)-- (5,-4);
\draw [line width=0.8pt] (7,4)-- (7,-4);
\draw [->,line width=0.8pt] (-2,-5) -- (-2,5);
\draw (1,-4.5) node {$\mathbf{\kappa_1}$};
\draw (3,-4.5) node {$\mathbf{\kappa_2}$};
\draw (5,-4.5) node {$\mathbf{\kappa_3}$};
\draw (7,-4.5) node {$\mathbf{\kappa_4}$};
\draw [->,line width=0.8pt] (5,1) -- (7,1);
\draw (-1.5,-2.5) node {$\mathbf{t_1}$};
\draw (-1.5,-0.5) node {$\mathbf{t_2}$};
\draw (-1.5,1.5) node {$\mathbf{t_3}$};
\draw (-3,5) node {$\mathbf{time}$};
\draw [->,line width=0.8pt] (3,-1) -- (1,-1);
\draw [->,line width=0.8pt] (3,-3) -- (5,-3);
\draw (5.5,-2.5) node {$\mathbf{\kappa_2}$};
\draw (0.5,-0.5) node {$\mathbf{\kappa_2}$};
\draw (7.5,1.5) node {$\mathbf{\kappa_2}$};
\draw (1,4.5) node {$\mathbf{\kappa_2}$};
\draw (3,4.5) node {$\mathbf{\kappa_2}$};
\draw (5,4.5) node {$\mathbf{\kappa_2}$};
\draw (7,4.5) node {$\mathbf{\kappa_2}$};
\begin{scriptsize}
\draw [color=black] (-2,1)-- ++(-2.5pt,-2.5pt) -- ++(5.0pt,5.0pt) ++(-5.0pt,0) -- ++(5.0pt,-5.0pt);
\draw [color=black] (-2,-1)-- ++(-2.5pt,-2.5pt) -- ++(5.0pt,5.0pt) ++(-5.0pt,0) -- ++(5.0pt,-5.0pt);
\draw [color=black] (-2,-3)-- ++(-2.5pt,-2.5pt) -- ++(5.0pt,5.0pt) ++(-5.0pt,0) -- ++(5.0pt,-5.0pt);
\end{scriptsize}
\end{tikzpicture}
\caption{\label{fig.1} \footnotesize Evolution of the families under neutrality; At time $0$ each individual represents its own family. At time 
$t_1$ the individual represented by $\kappa_3$ is replaced by an offspring of the  individual labeled by $\kappa_2$ and hence the family size represented by the individual with label $\kappa_2$ changes from 
$1/4$ to $1/2$ and the family size of the family represented by the individual with label $\kappa_3$ changes from $1/4$ to $0$.}
\end{center}
\end{figure}
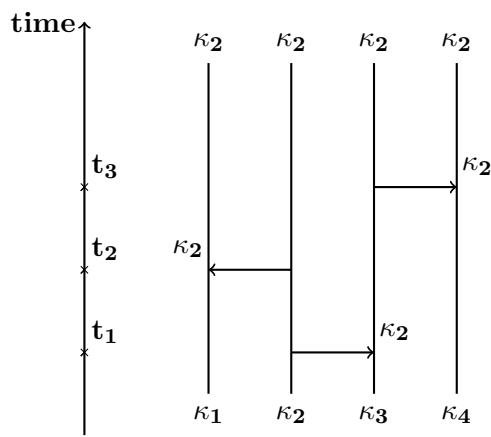

Note that this dynamic is the dynamic of an $N$-type Moran model with population size $N$ and the probability of sampling two individuals from 
the same family with frequency $\mathfrak{f}_i$ is (for large $N$)  given by $\mathfrak{f}_i^2$, or, to be more precise,  
\begin{equation}
P(R^N_t\le h) = \frac{1}{N^2} \sum_{i,j = 1}^N P(r_t(i,j) \le h) \approx \sum_{i = 1}^N E[\mathfrak{f}_i(h)^2]. 
\end{equation}

This is the situation under neutrality. When we now add mutation and selection, then we can use the same idea, but in contrast to the neutral model, the fitness of a family depends on the number of fit individuals within this family and mutation will change the type of an individual but not the family the individual belongs to. Roughly speaking we have external "types" with frequencies $\mathfrak{f}_1,\ldots,\mathfrak{f}_N$, which give the 
sizes of the different families and internal types with frequencies $y_1,\ldots,y_N$, which give the fraction of fit (i.e. of type $1$) types in this families (see Figure \ref{fig.2}).  

\begin{figure}[ht!]
\begin{center}
\begin{tikzpicture}[scale = 0.95]
\fill[line width=0.8pt,color=gray,fill=gray,fill opacity=0.1] (4,10) -- (4,8.8) -- (6,8.8) -- (6,10) -- cycle;
\fill[line width=0.8pt,color=gray,fill=gray,fill opacity=0.1] (9,10) -- (9,7.4) -- (11,7.4) -- (11,10) -- cycle;
\fill[line width=0.8pt,color=gray,fill=gray,fill opacity=0.1] (14,10) -- (14,8.2) -- (16,8.2) -- (16,10) -- cycle;
\draw [line width=0.8pt] (4,10)-- (4,7);
\draw [line width=0.8pt] (4,7)-- (6,7);
\draw [line width=0.8pt] (6,7)-- (6,10);
\draw [line width=0.8pt] (6,10)-- (4,10);
\draw [line width=0.8pt] (9,10)-- (9,7);
\draw [line width=0.8pt] (9,7)-- (11,7);
\draw [line width=0.8pt] (11,7)-- (11,10);
\draw [line width=0.8pt] (11,10)-- (9,10);
\draw [line width=0.8pt] (14,10)-- (14,7);
\draw [line width=0.8pt] (14,7)-- (16,7);
\draw [line width=0.8pt] (16,7)-- (16,10);
\draw [line width=0.8pt] (16,10)-- (14,10);
\draw [line width=0.8pt] (4,8.8)-- (6,8.8);
\draw [line width=0.8pt] (9,7.4)-- (11.,7.4);
\draw [line width=0.8pt] (14,8.2)-- (16,8.2);
\draw (5,6) node  {$\mathfrak{f}_1$};
\draw (10,6) node  {$\mathfrak{f}_2$};
\draw (15,6) node  {$\mathfrak{f}_3$};
\draw [->,line width=0.8pt] (4.7,9.1) -- (4.7,7.8);
\draw [->,line width=0.8pt] (5.3,7.8) -- (5.3,9.1);
\draw [->,line width=0.8pt] (6,9) -- (9,9);
\draw [->,line width=0.8pt] (9,8) -- (6,8);
\draw [->,line width=0.8pt,dash pattern=on 2pt off 2pt] (10.6,7.2) -- (14.4,8.8);
\draw [->,line width=0.8pt,dash pattern=on 2pt off 2pt] (10.6,9.4) -- (14.4,7.6);
\draw [->,line width=0.8pt,dash pattern=on 2pt off 2pt] (10.6,9.4) -- (14.4,9.4);
\draw [->,line width=0.8pt,dash pattern=on 2pt off 2pt] (10.6,7.2) -- (14.4,7.2);
\draw [line width=0.8pt] (12.25,8.15) -- (12.45,7.8);
\draw [line width=0.8pt] (12.35,8.2) -- (12.55,7.85);
\draw [line width=0.8pt] (12.45,7.4) -- (12.45,7);
\draw [line width=0.8pt] (12.55,7.4) -- (12.55,7);
\end{tikzpicture}
\end{center}
\caption{\label{fig.2}\footnotesize Here we have three families with relatively sizes $\mathfrak{f}_1,\mathfrak{f}_2,\mathfrak{f}_3$. The gray box indicates the fraction 
of fit types, which we denote by $y_1,y_2,y_3$. Mutation (as in the $\mathfrak{f}_1$ box) can only change the type within a family. 
Selection (dashed arrows from the $\mathfrak{f}_2$ box to the $\mathfrak{f}_3$ box) can only occur from a fit type. Resampling (black arrows between the 
$\mathfrak{f}_1$ and $\mathfrak{f}_2$ box) occurs independently of the internal type.}
\end{figure}
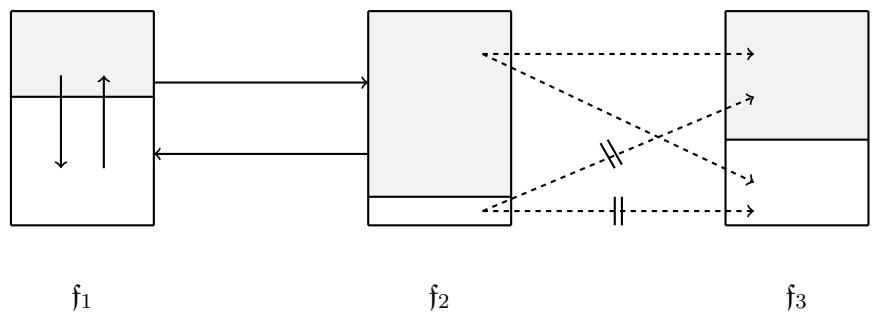

Instead of considering the external type frequency $\mathfrak{f}_i$ and the internal type frequency $y_i$ we will in the following consider the two internal type frequencies 
$y_i$ and $z_i$, where $z_i$ is the fraction of individuals with the unfit type in family $\kappa_i$, i.e. $z_i = \mathfrak{f}_i - y_i$. In the next section we give the formal statement of the above observations, where we identify the vector $(y_1,\ldots,y_N,z_1,\ldots,z_N)$ with a vector $(x_1,\ldots,x_{2N})$.

\subsection{Formal statement}\label{sec.formal}
Let 
\begin{equation}\label{eq.def.S}
\Simpl^{n} := \left\{\underline{x} \in [0,1]^{2n-1}:\ \sum_{i = 1}^{2n -1} x_i\le 1\right\},
\end{equation}
for some $n \in \mathbb N$. Moreover, let $\phi: \Simpl^n \rightarrow \mathbb R$ be a twice continuously differentiable function (in each component) and set $x_{2n} := 1-\sum_{i=1}^{2n-1} x_i$ for $\underline{x} \in \Simpl^n$. We define the following operator
\begin{equation}
L_\alpha \phi(\underline{x}) =  L^{\text{res}}\phi(\underline{x}) +  L^{\text{sel}}_\alpha\phi(\underline{x})+  L^{\text{mut}}\phi(\underline{x}),
\end{equation}
where for $\alpha, \ \vartheta_0,\vartheta_1\ge 0$:
\begin{align}
L^{\text{res}}\phi(\underline{x})= \frac{1}{2} &\sum_{i,j = 1}^{2n-1}  x_i(\delta_{i,j}-x_j) \ \frac{\partial}{\partial x_i}\frac{\partial}{\partial x_j}  \phi (\underline{x}),
\end{align}
\begin{align}
L^{\text{sel}}_\alpha\phi(\underline{x})= \alpha \left[\left(\sum_{i = 1}^{n} x_{i+n}\right) \sum_{i = 1}^n x_i \ \frac{\partial}{\partial x_i} \phi (\underline{x}) - \left(\sum_{i = 1}^{n} x_{i}\right) \sum_{i =n+1}^{2n-1} x_{i} \ \frac{\partial}{\partial x_i} \phi (\underline{x})\right]
\end{align}
and
\begin{align}
L^{\text{mut}}&\phi(\underline{x})=  \sum_{i = 1}^n  (\vartheta_0 x_{i+n} - \vartheta_1 x_{i}) \ \frac{\partial }{\partial x_i} \phi (\underline{x}) 
+  \sum_{i = 1}^{n-1}  (\vartheta_1 x_i-\vartheta_0 x_{i+n})\ \frac{\partial}{\partial x_{i+n}} \phi (\underline{x}).
\end{align}

\begin{proposition}  \label{lem.WFMMS}
The martingale problem associated with $(L,C^2(\Simpl^n),\underline{x})$ is well-posed for all initial values $\underline{x} \in \Simpl^n$. Moreover, the associated process $X^n$ is Feller (i.e. the corresponding semigroup is Feller) and has a modification with continuous paths. 
\end{proposition}

\begin{remark}
In the situation of the previous section $x_1,\ldots,x_n$ would be the fraction of fit types, i.e. $y_1,\ldots,y_n$, and $x_{n+1},\ldots,x_{2n}$  would be the fraction of unfit types, i.e. $z_1,\ldots,z_n$. 
\end{remark}

In the following we will always write for simplicity $Y^n(i) := X^n(i)$ and $Z^n(i):= X^n(i+n)$ for $i = 1,\ldots,n$, where $X^n(2n) := 1- \sum_{i = 1}^{2n-1} X^n(i) $. 

{
\begin{theorem}\label{thm.key}
Let $\bar Y_t $ be the two type Wright-Fisher diffusion at time $t > 0$ (given in the introduction) and let 
\begin{equation}
\left(Y^{n,t}_0(i),Z^{n,t}_0(i)\right) = \left(\frac{\bar Y_t}{n},\frac{1-\bar Y_t}{n}\right),\qquad i = 1,\ldots,n. 
\end{equation} 
\nolinebreak
Then 
\nolinebreak
\begin{equation}
P(R_{T} \le h) = \lim_{n \rightarrow \infty} E\left[\sum_{i = 1}^n \left(Y_{h}^{n,T-h}(i) + Z_{h}^{n,T-h}(i)\right)^2\right],
\end{equation}
\nolinebreak
for all $h < T$.
\end{theorem}
}
\begin{remark}
Note that we have
\begin{equation}
\left(\sum_{i = 1}^n Y^{n,T}_t(i)\right)_{t \ge 0} \stackrel{d}{=} \left(\bar Y_{t+T}\right)_{t \ge 0}.
\end{equation} 
\end{remark}

In order to avoid confusion we note at this point that we typically use $n$ for the dimension of the Wright-Fisher diffusion and $N$ for the population size of the finite model.

\section{Proofs}

Here we present the proofs of our results. We start with the proof of the main Theorem \ref{thm.main} and then we will prove the key observation, Theorem \ref{thm.key}. The proofs are based on the ideas in \cite{Grieshammer}.

\subsection{Proof of the main result}\label{sec.proof.main}

We use the notations from Section \ref{sec.formal}, where we abbreviate $X^n_t:=X^{n,T-h}_t$, $t \ge 0$, for some fixed $T > h > 0$ (at the end we are interested in $X^n_h$). We start with the following Lemma, which follows by a straight forward calculation (see Appendix \ref{a.gen}): 

\begin{lemma}
Define the functions 
\begin{align}
\phi_1: \Simpl^n \rightarrow \mathbb R,\quad  \underline{x} &\mapsto \sum_{i = 1}^n(x_i+x_{i+n})^2,\\
\phi_2: \Simpl^n \rightarrow \mathbb R,\quad  \underline{x} &\mapsto \sum_{i = 1}^n(x_i+x_{i+n})\left(x_i-(x_i+x_{i+n}) \sum_{j = 1}^n x_j\right) \\
\phi_3: \Simpl^n \rightarrow \mathbb R,\quad  \underline{x} &\mapsto 
\sum_{i = 1}^n \left(x_i-(x_i+x_{i+n}) \sum_{j = 1}^n x_j \right)^2 - 2 \phi_2(\underline{x})\cdot \sum_{j = 1}^n x_j 
\end{align}
where  $x_{2n} := 1-\sum_{i = 1}^{2n-1} x_i$. Then 
\begin{align}
L_\alpha \phi_1(\underline{x}) &= 1-\phi_1(\underline{x}) + 2 \alpha \phi_2(\underline{x}), \\
L_\alpha \phi_2(\underline{x}) &= -(3+\vartheta_0+\vartheta_1-\alpha) \phi_2(\underline{x}) + \alpha \phi_3(\underline{x}).
\end{align}
\end{lemma}

We define the following times: 
\begin{align}
t_0 &:= \inf\{t \ge 0:\ E[\phi_2(X^n_t)] < 0\},\\
\overline{\tau}_0 &:= \inf\{t \ge t_0:\ \phi_2(X^n_t) \ge 0\},\\
\underline{\tau}_0 &:= \inf\{t \ge t_0:\ \phi_2(X^n_t) < 0\}, \\ 
\underline{\tau}_\eps &:= \inf\{t \ge t_0:\ \phi_2(X^n_t) \le \eps\},
\end{align}
for $\eps > 0$.

\begin{lemma}\label{lem.sel.stop}
If we denote by $\mathcal F_t$ the filtration generated by $X^n$, then $\overline{\tau}_0,\underline{\tau}_0,\underline{\tau}_\eps$ are $\mathcal F_{t+} := \bigcap_{s > t} \mathcal F_{s}$-stopping times  and they satisfy  $\underline{\tau}_\eps \downarrow \underline{\tau}_0$ almost surely. Moreover, we have $\phi_2(X^n_{\overline{\tau}_0}) = \phi_2(X^n_{\underline{\tau}_0})=  E[\phi_2(X^n_{t_0})] = 0$ almost surely. 
\end{lemma}

\begin{proof}
The first part is Proposition 2.1.5 in \cite{EK86} together with the fact that $t \mapsto \phi_2(X^n_{t})$ is continuous (note that the image of a compact interval under a continuous function is compact and therefore closed). \par For the second part observe that 
$\phi_2(X^n_{t}) \le -\eps$ implies $\phi_2(X^n_{t}) \le -\eps'$ implies $\phi_2(X^n_{t}) < 0$ for all $0 < \eps' \le \eps$, by continuity. This gives $\underline{\tau}_0  \le \underline{\tau}_{\eps'} \le \underline{\tau}_\eps$. Now, by definition of 
$\underline{\tau}_0$, there is a sequence $t_n \ge \tau_0$ such that $t_n \downarrow \tau_0$ and $\phi_2(X^n_{t_n}) \le -\eps_n$ for some 
$\eps_n > 0$ and all $n \in \mathbb N$. This gives $t_n \ge \tau_{\eps_n}$ and the second part follows. \par
The last part is again a consequence of the continuity. 
\end{proof}

\begin{lemma}\label{lem.sel.end}
There is an almost surely finite stopping time $S$ such that $\phi_2(X^n_{t}) = 0$ for all $t \ge S$ and $S$ is integrable.  
\end{lemma}

\begin{proof}
We will prove this Lemma in Section \ref{sec.p.key}. 
\end{proof}

We are now ready to prove our main result and split the proof in two steps. First we prove the dominance and then we show convergence when $\alpha \rightarrow\infty$. \\

\noindent {\bf Proof of Theorem \ref{thm.main} - stochastic dominance.} \\

We start by showing $ E[\phi_2(X^n_t)] \ge 0$ for all $t \ge 0$ and all $n$ large enough. 

\begin{lemma}\label{lem.t0}
One has $t_0 > 0$ for all $n$ large enough. 
\end{lemma}

\begin{proof}
Recall, $X^n_0 = X^{n,T-h}_0$ and the initial condition given in Theorem \ref{thm.key}. Then we have 
\begin{equation}
\begin{split}
\phi_2(X^n_{0}) = \sum_{i = 1}^n \frac{1}{n} \left(\frac{\bar Y_{T-h}}{n} - \frac{1}{n} \cdot \bar Y_{T-h} \right)= 0.
\end{split}
\end{equation}
Since 
\begin{align}
\frac{d}{dt}   [\phi_2(X^n_t)] &=  E[L_\alpha \phi_2 (X^n_t)],\\
\frac{d^2}{dt^2}  E[\phi_2(X^n_t)] &= E[L_\alpha(L_\alpha \phi_2)(X^n_t)],
\end{align}
 a generator calculation (see Appendix \ref{a.gen}) shows that 
\begin{equation}
\begin{split}
\frac{d}{dt}   E[\phi_2(X^n_t)]\Big|_{t = 0} = 0
\end{split}
\end{equation}
and
\begin{equation}
\begin{split}
\frac{d^2}{dt^2}  E[\phi_2(X^n_t)] \Big|_{t = 0} = \alpha E[&\bar Y_{T-h} (1-\bar Y_{T-h})] - \frac{\alpha }{n} ( E[\bar Y_{T-h} (1-\bar Y_{T-h})] )  \\
& - \frac{6\alpha^2}{n}  E[\bar Y_{T-h}^2 (1-\bar Y_{T-h})^2(4\bar Y_{T-h} - 3)] . 
\end{split}
\end{equation}
In the case, where $\vartheta_0\wedge \vartheta_1 > 0$, $\mathcal L(\bar Y_t)$ is absolutely continuous to the Lebesgue measure and therefore $ E[\bar Y_{T-h} (1-\bar Y_{T-h})] > 0$ which gives the result. \par 
In the case where either $\vartheta_0=0$ or $\vartheta_1=0$
and the process $\bar Y_t$ starts in its absorbing state, we have $\bar Y_t = \bar Y_0 \in \{0,1\}$ for all $t \ge 0$, which implies $X^n_t(i)\cdot X^n_t(i+n) = 0$ for all $i = 1,\ldots,n$. This gives   
$ E[\phi_2(X^n_t)] = 0$ for all $t \ge 0$ and therefore $t_0 = \infty$.  If $\bar Y_t$ does not start in its absorbing state, then
$P(\bar Y_t \not\in\{0,1\}) > 0$ and the result follows analogue to the above. 
\end{proof}

In order to show $ E[\phi_2(X^n_t)] \ge 0$, we abbreviate $g_t := \phi_2(X^n_t)$ and note that $ E[\phi_2(X^n_t)] \ge 0$ is equivalent to $t_0 = \infty$. Assume that $t_0 < \infty$. Then  $S_0 := S \vee t_0$ is an integrable stopping time and
\begin{equation}
\begin{split}
 E\left[g_{\overline{\tau}_0 \wedge \underline{\tau}_\eps \wedge S_0} - g_{t_0}\right] = 
 E\left[M^g_{\overline{\tau}_0 \wedge \underline{\tau}_\eps \wedge S_0} -M^g_{t_0}\right]  +  E\left[\int_{t_0}^{\overline{\tau}_0 \wedge \underline{\tau}_\eps \wedge S_0} L_\alpha g_s ds\right],  
\end{split}
\end{equation}
where
\begin{equation}
M^g_t := g_t - g_0 - \int_0^t L_\alpha\phi_2 (X^n_s) ds 
\end{equation}
is a continuous bounded martingale (see Proposition \ref{lem.WFMMS}).
Note that $M^g$ is also a martingale with respect to $\{\mathcal F_{t+}\}$ (see Remark 2.2.14 in \cite{EK86}) and by the optional sampling theorem (see Theorem 2.2.13 in \cite{EK86})  we get 
\begin{equation}
E\left[M^g_{\overline{\tau}_0 \wedge \underline{\tau}_\eps \wedge S_0}\right] =  E\left[M^g_{t_0}\right] = 0.
\end{equation}
It follows that 

\begin{equation}
\begin{split}
  E\left[g_{\overline{\tau}_0 \wedge \underline{\tau}_\eps \wedge S_0} - g_{t_0}\right] = 
  E\left[g_{\underline{\tau}_\eps } 1(\underline{\tau}_\eps \le \overline{\tau}_0  \wedge S_0)\right] 
= - \eps  P\left(\underline{\tau}_\eps \le \overline{\tau}_0  \wedge S_0\right)
\end{split}
\end{equation}
and 
\begin{equation}
\begin{split}
  E&\left[\int_{t_0}^{\overline{\tau}_0 \wedge \underline{\tau}_\eps \wedge S_0} L_\alpha g_s ds\right] \\
&=  E\left[\int_{t_0}^{\overline{\tau}_0 \wedge \underline{\tau}_\eps \wedge S_0} -(3+\vartheta_0+\vartheta_1-\alpha) \phi_2(X^n_s) + \alpha \phi_3(X^n_s) ds\right] \\
&\ge  E\left[\int_{t_0}^{\overline{\tau}_0 \wedge \underline{\tau}_\eps \wedge S_0} -(3+\vartheta_0+\vartheta_1-\alpha) \phi_2(X^n_s) \phantom{\sum_{i = 1}^n X^n_s(i)}\right.\\
&{}\hspace{5cm} \left.- 2 \alpha \phi_2(X^n_s)\cdot \sum_{i = 1}^n X^n_s(i) ds\right] \\
&=   E\left[\int_{t_0}^{\overline{\tau}_0 \wedge \underline{\tau}_\eps \wedge S_0} -(3+\vartheta_0+\vartheta_1+\alpha) \phi_2(X^n_s) \phantom{\left(1-\sum_{i = 1}^n X^n_s(i)\right)}\right.\\
&{}\hspace{5cm}\left.+ 2 \alpha \phi_2(X^n_s)\cdot \left(1-\sum_{i = 1}^n X^n_s(i)\right) ds\right] \\
&\ge  E\left[\int_{t_0}^{\overline{\tau}_0 \wedge \underline{\tau}_\eps \wedge S_0} 2 \alpha (-\eps) \cdot \left(1-\sum_{i = 1}^n X^n_s(i)\right) ds\right] \\
&\ge -\eps \cdot 2\alpha E\left[\overline{\tau}_0 \wedge \underline{\tau}_\eps \wedge S_0 - t_0\right]. 
\end{split}
\end{equation}
Therefore, by Lemma \ref{lem.sel.stop} and Lemma \ref{lem.sel.end}, 

\begin{equation}
\begin{split}
P\left(\underline{\tau}_0 < \overline{\tau}_0  \wedge S_0\right) &\le \lim_{\eps \downarrow 0} P\left(\underline{\tau}_\eps \le \overline{\tau}_0  \wedge S_0\right) \\
&\le 2\alpha \lim_{\eps \downarrow 0}  E\left[\overline{\tau}_0 \wedge \underline{\tau}_\eps \wedge S_0 - t_0\right] \\
&= 2 \alpha  E\left[\overline{\tau}_0 \wedge \underline{\tau}_0 \wedge S_0 - t_0\right]. 
\end{split}
\end{equation}
Now, by definition, $\overline{\tau}_0 \wedge \underline{\tau}_0 = t_0$ and hence 
\begin{equation}
\begin{split}
P\left(\underline{\tau}_0 < \overline{\tau}_0  \wedge S_0\right) &\le 2 \alpha E\left[t_0 \wedge (S \vee t_0) - t_0\right] = 0. 
\end{split}
\end{equation}
Note that by Lemma \ref{lem.sel.end} and the definition of $\underline{\tau}_0$ we have $\{S \le \underline{\tau}_0\} = \{\underline{\tau}_0 = \infty\}$. Hence, 

\begin{equation}
\begin{split}
P\left(\overline{\tau}_0 \le \underline{\tau}_0\right) &= 
P\left(\{\overline{\tau}_0 \le \underline{\tau}_0\} \cup \{\underline{\tau}_0 = \infty\}\right) \\
&= P\left(\{\overline{\tau}_0 \le \underline{\tau}_0\} \cup \{S \le \underline{\tau}_0\}\right) \\
&= P\left(\overline{\tau}_0  \wedge (S \vee t_0) \le \underline{\tau}_0\right) = 1.
\end{split}
\end{equation}
This gives $ \phi_2(X^n_{t_0}) \ge 0$ almost surely and therefore, since $ E[\phi_2(X^n_{t_0})] = 0$, $ \phi_2(X^n_{t_0}) = 0$ almost surely. 

\begin{lemma}
Let $\vartheta_0 \wedge \vartheta_1 > 0$, then $P(\phi_2(X^n_{t}) = x) = 0$ for all $x \in \mathbb R$ and all $t > 0$. In the 
case, where $\vartheta_0 = 0$ or $\vartheta_1 = 0$, we either have $P(\phi_2(X^n_{t}) = 0) = 1 $ for all $t > 0$ or 
 $P(\phi_2(X^n_{t}) = 0) < 1 $ for all $t > 0$. 
\end{lemma}

\begin{proof}
The first part follows by the fact that $\mathcal L(X^n_t)$ is absolutely continuous to the Lebesgue measure. The second part follows similarly since $P(\bar Y_t \in \{0,1\}) < 1$, with the one exception, that when the Wright-Fisher diffusion $\bar Y_t$ starts in its absorbing state, then 
$\bar Y_t = \bar Y_0 \in \{0,1\}$ for all $t \ge 0$ which implies $P(\phi_2(X^n_{t}) = 0) = 1 $ for all $t > 0$.   
\end{proof}

In view of the above Lemma $ \phi_2(X^n_{t_0}) = 0$ almost surely is only possible if either $t_0 = 0$, which would contradict Lemma \ref{lem.t0}, or $t \mapsto E[\phi_2(X^n_{t})] \equiv 0$. Hence $t_0 = \infty$ and therefore 
\begin{equation}
  E[\phi_2(X^n_{t})] \ge 0 \qquad \text{for all } t \ge 0 \text{ and all } n \text{ sufficiently large}. 
\end{equation}

Recall that 
\begin{equation}
(X^{n}_0(i) + X_{0}^{n}(i+n)) = \frac{1}{n}.
\end{equation}
If we now write $X^{n,\alpha}_{t}$ in order to indicate the dependence on the selection parameter $\alpha \ge 0$, then 
\begin{equation}
\frac{d}{dt} E[\phi_1(X_t^{n,\alpha})] = 1-   E[\phi_1(X_t^{n,\alpha})] + 2\alpha  E[\phi_2(X_t^{n,\alpha})] , 
\end{equation}
which implies 
\begin{equation}
\begin{split}
  E[\phi_1(X_t^{n,\alpha})]  &= \frac{1}{n} e^{- t} + ( 1-e^{-t}) +  2\alpha e^{-t} \int_0^t e^{s}  E[\phi_2(X_s^{n,\alpha})]ds \\
&\ge \frac{1}{n} e^{- t} + ( 1-e^{-t})  = E[\phi_1(X_t^{n,0})],
\end{split}
\end{equation}
for all $n$ sufficiently large. We can now apply Theorem \ref{thm.key} to get

\begin{equation}
\begin{split}
P(R^\alpha_{T} \le h)&= \lim_{n \rightarrow \infty}  E[\phi_1(X^{n,T-h,\alpha}_{h})]  \\
&\ge 1-e^{-h} \\
&= \lim_{n \rightarrow \infty}  E [\phi_1(X^{n,T-h,0}_{h})]  = P(R^0_{T} \le h),
\end{split}
\end{equation}
for all $h < T$. Since $P(R^\alpha_{T} \le h) = P(R^0_{T} \le h) = 1$ for all $h \ge T$ a classical result of Strassen (see \cite{L99} or \cite{Strassen}) gives the stochastic dominance. \\

\noindent {\bf Proof of Theorem \ref{thm.main} for $\alpha \rightarrow \infty$.} \\

Recall that the total number of fit types $\bar Y_T = \sum_{i = 1}^n X_h^{n,T-h}(i)$ is a Wright-Fisher diffusion  with 
generator 
\begin{equation}
Gf(x) := \frac{1}{2} x(1-x) \frac{\partial^2}{\partial x^2} f(x) + (-\vartheta_1 x + \vartheta_0 (1-x) + \alpha x (1-x))  \frac{\partial}{\partial x} f(x).
\end{equation}
Moreover, recall that, by Remark \ref{r.WF.equi}, the stationary distribution has the following density: 
\begin{equation}
\rho(x) = \frac{x^{2\vartheta_0 - 1} (1-x)^{2\vartheta_1 - 1} \exp(2 \alpha x)}{\int_0^1 x^{2\vartheta_0 - 1} (1-x)^{2\vartheta_1 - 1} \exp(2 \alpha x) dx}.
\end{equation}

\begin{lemma}\label{l.WF.moments}
Let $\bar Y$ be a $[0,1]$-valued random variable with density $\rho$, where we assume that $\vartheta_0 = \vartheta_1 =: \vartheta$. 
Then
\begin{align}
\alpha \cdot E[1-\bar Y] &\rightarrow \vartheta,\\
\alpha^2 \cdot E[(1-\bar Y)^2] &\rightarrow \vartheta^2+ \frac{1}{2} \vartheta,
\end{align}
when $\alpha \rightarrow \infty$. In fact, we have $\alpha\cdot E[(1-\bar Y)^2] \rightarrow 0$.
\end{lemma}

\begin{proof}
One can easily see that for $\vartheta = 1/2$,
\begin{align}
E[1-\bar Y] &= \frac{1}{2}\left(\frac{1}{\alpha} - \frac{2}{e^{2\alpha}-1}\right),\\
E[(1-\bar Y)^2] &= \frac{1}{2}\left(\frac{1}{\alpha^2} - \frac{2(\alpha+1)}{\alpha(e^{2\alpha}-1)}\right).
\end{align}
For arbitrary $\vartheta$ one can use numerical software such as MAPLE in order to show the result. 
\end{proof}

\begin{remark}\label{r.gen.mut}
In order to generalize the result to arbitrary mutation parameters $\vartheta_0,\vartheta_1$ one needs to verify $\alpha E[(1-\bar Y)^2] \rightarrow 0$, when $\alpha \rightarrow \infty$.
\end{remark}

\noindent Now observe that 
{\small
\begin{equation}
\begin{split}
\phi_2&(X^n_t) \\
&=  \sum_{i = 1}^n \left(X^n_t(i)+X^n_t(i+n)\right)\left(X^n_t(i)-\left(X^n_t(i)+X^n_t(i+n)\right) \sum_{j = 1}^n X^n_t(j)\right) \\
&\le  \left(1 - \sum_{j = 1}^n X^n_t(j)\right) \cdot \left(\sum_{i = 1}^n X^n_t(i)+X^n_t(i+n)\right)^2 \\
&= 1 - \sum_{j = 1}^n X^n_t(j) \stackrel{d}{=} 1-\bar Y_T
\end{split}
\end{equation}
and 
\begin{equation}
\begin{split}
\sum_{i = 1}^n  &\left(X^n_t(i)-(X^n_t(i)+X^n_t(i+n)) \sum_{j = 1}^n X^n_t(j)\right)^2 \\
&= \sum_{i = 1}^n  \left(1\Big(X^n_t(i)\ge (X^n_t(i)+X^n_t(i+n)) \sum_{j = 1}^n X^n_t(j)) \right. \\
&{}\hspace{2cm}\left. + 1(X^n_t(i)<(X^n_t(i)+X^n_t(i+n)) \sum_{j = 1}^n X^n_t(j)\Big) \right) \cdot \\
&{}\hspace{4cm}\cdot\left(X^n_t(i)-(X^n_t(i)+X^n_t(i+n)) \sum_{j = 1}^n X^n_t(j)\right)^2 \\
&\le \left(1 - \sum_{j = 1}^n X^n_t(j)\right)^2\cdot \sum_{i = 1}^n \left(X^n_t(i)+X^n_t(i+n)\right)^2  
 + \sum_{i = 1}^n X^n_t(i+n)^2 \\
&\le  \left(1 - \sum_{j = 1}^n X^n_t(j)\right)^2 + \left(\sum_{i = 1}^n X^n_t(i+n)\right)^2  \\
&\stackrel{d}{=} 2 \left(1-\bar Y_T\right)^2. 
\end{split}
\end{equation}
Note that $E[\left(1-\bar Y_T\right)^2]$ is independent of $T$ and given in Lemma \ref{l.WF.moments}. 
If we denote by $f_\alpha, g_\alpha$ the solutions of the equations
\begin{align}
\frac{d}{dt}  f_\alpha(t) &= 1-  f_\alpha(t) + 2\alpha g_\alpha(t),\qquad f_\alpha(0) = \frac{1}{n}, \\
\frac{d}{dt}  g_\alpha(t) &= -(3+2\vartheta +\alpha) g_\alpha(t) + 4 \alpha E[\left(1-\bar Y_0\right)^2],\qquad g_\alpha(0) =0, 
\end{align}
then 

\begin{equation}
\lim_{n \rightarrow \infty} f_\alpha(h) = (1-e^{-h}) + 4 (\vartheta + 2 \vartheta^2)\frac{1}{\alpha} (1-e^{-h}) + \mathcal O(1/\alpha^2),
\end{equation}
for $\alpha \rightarrow \infty$.  Hence we have

\begin{equation}
\begin{split}
\lim_{\alpha \rightarrow \infty} \lim_{n \rightarrow \infty} E_\alpha[\phi_1(X^n_h)] &\le \lim_{\alpha \rightarrow \infty} \lim_{n \rightarrow \infty} f_\alpha(h) \\
&= \lim_{\alpha \rightarrow \infty}  \left(1- e^{-h} +4 \frac{1}{\alpha} (1-e^{-h})\right) \\
&= 1- e^{-h} 
\end{split}
\end{equation}
and the result follows analogue to the first step combined with the domination result.

\subsection{Proof of the key observation}

Here we prove our main tool. In Section \ref{sec.finite.population} we show how to relate a measure-valued process with 
the genealogical distance. Next, in Section \ref{sec.limit}, we give the large population limit of this measure-valued process.
The characterization of this limit will then be used in Section \ref{sec.p.key} to prove the result. 

\subsubsection{Family dynamic of the finite population model}\label{sec.finite.population}

We start with the notion of descendents.\\

\noindent{\bf Descendents}\\ 

\noindent	Let $M \subset I_N$, $0 \le h \le T$ and define 
	\begin{equation}
	D_{h,T}(M):= \{i \in I_N:\ A_h(i,T) \in M\}.
	\end{equation}
	We call $D_{h,T}(M)$ the set of {\it descendants} at time $T$ of the individuals in $M$ that lived at time $h$. 
	We abbreviate $D_{h,T}(x):=D_{h,T}(\{x\})$ and write $\mathcal A_{h,T}:= \{i \in I_N: D_{h,T}(i) \neq \emptyset\} $ for the set of ancestors at time $T-h$ (measured backward).  
	We note that $ D_{h,T}(i) = B^{r_{T}}_{T-h}(j) = \{k \in I_N:\ r_T(k,j) \le T-h\}$ is a closed ball of radius $T-h$ for all $i \in I_N$ and $j \in  D_{h,T}(i)$. It follows that 
\begin{equation}
I_N = \biguplus_{i \in \mathcal A_{h,T}}  D_{h,T}(i)= \biguplus_{i \in I_N}  D_{h,T}(i)
\end{equation}
is the disjoint union of closed  balls (with respect to $r_{T}$) with radius $T-h$. \\

\noindent{\bf The model under neutrality} \\

\noindent	We define for $0\le t,T$
\begin{equation}
\begin{split}
\mathcal X_t^{N,T}&:= \sum_{i \in \mathcal A_{T,T+t}} \frac{|D_{T,T+t}(i)|}{N} \delta_{\kappa_i} = \sum_{i \in I_N} \frac{|D_{T,T+t}(i)|}{N} \delta_{\kappa_i},
\end{split}
\end{equation}
where we assume in the following 
\begin{itemize}
\item[] {\it $(\kappa_i)_{i \in \mathbb N}$ are i.i.d. uniformly distributed random variables also independent of the random mechanisms given in Section \ref{sec.construction}.}
\end{itemize}
Then 
\begin{equation}
\mathcal X_0^{N,T} = \frac{1}{N} \sum_{i \in I_N}  \delta_{\kappa_i} 
\end{equation}
and $(X_t^{N,T})_{t \ge 0}$ has the following dynamic:
Define $\xi_i(t):= \kappa_{A_T(T+t,i)}$, $i \in I_N$. Then $t \mapsto (\xi_i(t))_{i \in I_N}$ evolves as follows. 
	If $\eta^{i,j}_{\text{res}}(\{t+T\}) = 1$, then $A_T(j,T+t) = A_{T}(i,(T+t)-)$ and hence 
	\begin{equation}
  \xi_k(t) \rightarrow \left\{\begin{array}{ll}
	 \xi_i(t),&\quad \text{if } k = j,\\
	 \xi_k(t),&\quad \text{if } k\neq j. 
	\end{array}\right.
	\end{equation}
Therefore,
	\begin{equation}
	\begin{split}
  \frac{1}{N}\sum_{i \in I_N} &\delta_{\xi_i(t)}  =\sum_{i \in I_N} \sum_{j \in D_{T,T+t}(i)} \frac{1}{N} \delta_{\xi_j(t)} \\
	&=\sum_{i \in I_N} \sum_{j \in D_{T,T+t}(i)} \frac{1}{N}  \delta_{\kappa_i} = \mathcal X_t^{N,T}
	\end{split}
	\end{equation}
	satisfies the definition of a measure-valued Moran model (see for example \cite{dawson} or \cite{EK93}). Recall that the generator $\tilde \Omega^N$ of the measure-valued Moran model 
	is defined on the set of {\it polynomials} $\Pi$, where we call a function a polynomial if it is the linear combinations of functions of the form 
	\begin{equation}\label{eq.polyn}
\Phi = \langle \phi ,\mu^m\rangle := \int \phi \ d\mu^{\otimes m}, 
\end{equation} 
where $\phi:[0,1]^m \to [0,1]^m$ is continuous, $m \in \mathbb N$. 
\begin{remark}
Note that by the Stone-Weierstrass theorem $\Pi$ is dense in $C(\mathcal M_1([0,1]))$, when $\mathcal M_1([0,1])$ is equipped with the weak topology. 
 \end{remark}
For such functions $\Phi$ the operator $\tilde \Omega^N$ is given by 
\begin{equation}\label{eq.generator.neutral}
\tilde \Omega^N\Phi(\mu):=\Omega^N_{\text{res}} \Phi(\mu) := \binom{N}{2} \int_{[0,1]} \int_{[0,1]} \left( \Phi(\mu^{u,v}) - \Phi(\mu)\right) \mu(du) \mu(dv), 
\end{equation}
where 
\begin{equation}
\mu^{u,v} = \mu + \frac{1}{N} \delta_{u}- \frac{1}{N} \delta_{v}.
\end{equation}

\noindent{\bf Adding selection} \\

\noindent In order to describe the model with mutation and selection, we define the process $(\underline{Y},\underline{Z})= ((\underline{Y}_t,\underline{Z}_t))_{t \ge 0} = ((\underline{Y}_t^N,\underline{Z}_t^N))_{t \ge 0}$, where we interpret $Y_t(i)$ ($Z_t(i)$) as the relative number of fit (unfit) descends at time $t$ of individual $i$. Then $X_t^{N,T}(i) := Y_t(i) + Z_t(i)$ is the size of the $i$th family, i.e. the quantity we are interested in.  Observe that $(\underline{Y},\underline{Z})$  has the following dynamic: \\

At rate $\binom{N}{2}$ we have four different transitions (let $e_i(j) = 0$ for $j \neq i$ and $e_i(i) = 1$): 
\begin{equation}
(\underline{Y}_t,\underline{Z}_t) \rightarrow \left(\underline{Y}_t + \frac{1}{N}e_i  ,\underline{Z}_t- \frac{1}{N}e_j\right)
\end{equation}
with probability $Y_t(i)\cdot Z_t(j) $ (i.e. an unfit descendant $j$ is replaced by a fit descendant of $i$ due to resampling),
\begin{equation}
(\underline{Y}_t,\underline{Z}_t) \rightarrow \left(\underline{Y}_t + \frac{1}{N}e_i - \frac{1}{N}e_j ,\underline{Z}_t\right)
\end{equation}
with probability $Y_t(i)\cdot Y_t(j)$ (i.e. a fit descendant $j$ is replaced by a fit descendant of $i$ due to resampling),
\begin{equation}
(\underline{Y}_t,\underline{Z}_t) \rightarrow \left(\underline{Y}_t - \frac{1}{N}e_j ,\underline{Z}_t+\frac{1}{N} e_i\right)
\end{equation}
with probability $Z_t(i)\cdot Y_t(j)$ (i.e. a fit descendant $j$ is replaced by an unfit descendant of $i$ due to resampling),
\begin{equation}
(\underline{Y}_t,\underline{Z}_t) \rightarrow \left(\underline{Y}_t ,\underline{Z}_t+ \frac{1}{N}e_i - \frac{1}{N}e_j \right)
\end{equation}
with probability $Z_t(i)\cdot Z_t(j)$ (i.e. an unfit descendant $j$ is replaced by an unfit descendant of $i$ due to resampling). \\
 
At rate $\frac{\alpha}{N} \binom{N}{2}$ we have 
\begin{equation}
(\underline{Y}_t,\underline{Z}_t) \rightarrow \left(\underline{Y}_t + \frac{1}{N}e_i  ,\underline{Z}_t - \frac{1}{N}e_j\right)
\end{equation}
with probability $Y_t(i)\cdot Z_t(j)$ (i.e. an unfit descendant $j$ is replaced by a fit descendant of $i$ due to selection),
\begin{equation}
(\underline{Y}_t,\underline{Z}_t) \rightarrow \left(\underline{Y}_t + \frac{1}{N}e_i - \frac{1}{N}e_j ,\underline{Z}_t\right)
\end{equation}
with probability $Y_t(i)\cdot Y_t(j)$ (i.e. a fit descendant $j$ is replaced by a fit descendant of $i$ due to selection). \\
 
At rate $\vartheta_0$ we have 
\begin{equation}
(\underline{Y}_t,\underline{Z}_t) \rightarrow \left(\underline{Y}_t+\frac{1}{N} e_i,\underline{Z}_t-\frac{1}{N} e_i\right)
\end{equation}
with probability $Z_t(i)$ (i.e. mutation from an unfit descendent of $i$ to a fit descendent) and with rate $\vartheta_1$ we have
\begin{equation}
(\underline{Y}_t,\underline{Z}_t) \rightarrow \left(\underline{Y}_t-\frac{1}{N} e_i,\underline{Z}_t+\frac{1}{N} e_i\right)
\end{equation}
with probability $Y_t(i)$ (i.e. mutation from a fit descendent of $i$ to an unfit descendent). \\

Note that $(\underline{Y},\underline{Z})$ is a $(2N-1)$-dimensional Wright-Fisher model: 

\begin{lemma}\label{lem.WFMS.finite}
Let $\phi:[0,1]^N\times [0,1]^N\to \mathbb R$ be continuous. Then the generator $L^N$ of the Markov jump process 
$\underline{X}^N$ with $X_t^N(i) = Y^N_t(i)$, $X_t^N(i+N) = Z^N_t(i)$, $i = 1,\ldots, N$ is given by 

\begin{equation}
\begin{split}
L^N \phi(\underline{x}) = &\binom{N}{2} \left(\sum_{i,j = 1}^{2N} x_i x_j \phi(\underline{x} + \frac{1}{N}e_i - \frac{1}{N}e_j) -\phi(\underline{x})\right)\\
&+ \frac{\alpha}{N} \cdot \binom{N}{2}  \sum_{i,j = 1}^N \Big(x_i x_{j+N}  \phi(\underline{x}+1/N e_i-1/N e_{j+N}) \\
&{}\hspace{3.5cm}+ x_i x_j \phi(\underline{x}+1/N e_i-1/N e_{j})- \phi(\underline{x}) \Big) \\ 
&+  N \vartheta_1 \sum_{i = 1}^N x_i \left(  \phi(\underline{x}-1/N e_i +1/N e_{i+N})  - \phi(\underline{x}) \right)\\
&{} \hspace{1.5cm} + N \vartheta_0 \sum_{i = 1}^N x_{i+N}\left( \phi(\underline{x}+1/N e_i -1/N e_{i+N}) - \phi(\underline{x})  \right). 
\end{split}
\end{equation}
\end{lemma}

As before, we now want to define a suitable measure-valued process:

\begin{lemma}\label{lem.MMMS.generator}
Recall that $(\kappa_i)_{i \in \mathbb N}$ are i.i.d. uniformly $[0,1]$. Let 
\begin{equation}
\mathcal E:[0,1]^{2N} \mapsto \mathcal M_1([0,2]), \quad x \mapsto 
\sum_{i = 1}^N x(i)\delta_{\kappa_i} + \sum_{i = 1}^N x(i+N)\delta_{\kappa_i+1}
\end{equation}
and define $\mathcal X^{N,T}_t := \mathcal E(\underline{X}^N_t)$. Then its generator $\Omega^N$   
(recall \eqref{eq.generator.neutral}) is given by 
\begin{equation}
(L^N(\Phi\circ\mathcal E))(x) = (\Omega^N\Phi)(\mathcal E(x))
\end{equation}
and satisfies

\begin{equation}
\Omega^N \Phi(\mu) =  \Omega^N_{\text{res}} \Phi(\mu) +  \Omega^N_{\text{sel}} \Phi(\mu) +\Omega^N_{\text{mut}} \Phi(\mu)
\end{equation}
with 
\begin{equation}
\Omega^N_{\text{res}} \Phi(\mu) := \binom{N}{2} \int_{[0,2]} \int_{[0,2]} \left( \Phi(\mu^{u,v}) - \Phi(\mu)\right) \mu(du) \mu(dv), 
\end{equation}
\begin{equation}
\begin{split}
\Omega^N_{\text{sel}} \Phi(\mu) := \frac{\alpha}{N} \cdot \binom{N}{2} \Big( \int_{[0,2]} \int_{[0,2]} \left( \Phi(\mu^{u,v}) \chi(u) - \Phi(\mu) \right)\mu(du)\mu(dv),
\end{split}
\end{equation}
where $\chi(u) = 1(u \le 1)$ and 
\begin{equation}
\begin{split}
\Omega^N_{\text{mut}} \Phi(\mu) :=   N \Big(\vartheta_1\int_{[0,1]} &\left(\Phi(\Theta_u(\mu)) - \Phi(\mu)\right) \mu(du) \\
&{}\hspace{1cm}+  \vartheta_0\int_{(1,2]} \left(\Phi(\Theta_u'(\mu) - \Phi(\mu)\right)\mu(du)\Big),
\end{split}
\end{equation}
where 
\begin{equation}
\begin{split}
\mu^{u,v} &= \mu + \frac{1}{N} \delta_{u}- \frac{1}{N} \delta_{v}, \\
\Theta_u(\mu) &= \mu + \frac{1}{N} \delta_{u+1} - \frac{1}{N} \delta_{u},\\
\Theta_u'(\mu) &= \mu + \frac{1}{N} \delta_{u-1} - \frac{1}{N} \delta_{u}.
\end{split}
\end{equation}
\end{lemma}

\begin{proof}
This is a straight forward calculation.
\end{proof}

\begin{remark}\label{rem.con.MM.PW.2}
Note that 
\begin{equation}
\mathcal X^{N,T}_0  = \frac{1}{N}\sum_{i = 1}^N  u_i(T) \delta_{\kappa_i} +\frac{1}{N} \sum_{i = 1}^N (1-u_i(T))\delta_{\kappa_i+1}. 
\end{equation} 
\end{remark}

Finally observe the following two facts: 

\begin{lemma}\label{lem.dist.finite}
One has 
\begin{equation}
\begin{split}
 P(R_{T}^N \le h) &= E\left[\int_{[0,2]} \mathcal X^{N,T-h}_h(\{x\}) \mathcal X^{N,T-h}_h(dx)+\right. \\
&{}\hspace{2cm}\left. + 2 \int_{[0,1]}\mathcal X^{N,T-h}_h(\{x+1\})\mathcal X^{N,T-h}_h(dx) \right],
\end{split}
\end{equation}
for all $0 \le h \le T$.
\end{lemma}

\begin{proof}
One can follow the approach, presented in the previous section, in order to prove that 
\begin{equation}\label{eq.connection}
\left(\mathcal X^{N,T}_t(\{\kappa_i\}) + \mathcal X^{N,T}_t(\{\kappa_i+1\})\right)_{t \ge 0} \stackrel{d}{=} \left(\frac{|D_{T,T+t}(i)|}{N}\right)_{t \ge 0}. 
\end{equation}
By the definition of $D_{t,T}$ and $R_T^N$ we get that 
\begin{equation}
\begin{split}
P(R_T^N \le h) &= \frac{1}{N^2}\sum_{i,j \in I_N} P(r_T(i,j) \le h) \\
&= \frac{1}{N^2}\sum_{i,j \in I_N} \sum_{k \in I_N} E\left[1(i \in B_{T-h}^{r_T}(k)) 1(j \in B_{T-h}^{r_T}(k))\right] \\
&= \frac{1}{N^2}\sum_{i,j \in I_N} E\left[\sum_{k \in \mathcal A_{T-h,T}} 1(i \in D_{T-h,T}(k)) 1(j \in D_{T-h,T}(k))\right] \\
&= \frac{1}{N^2} E\left[\sum_{k \in \mathcal A_{T-h,T}} |D_{T-h,T}(k)|^2\right] \\
&= \frac{1}{N^2} E\left[\sum_{i \in I_N} |D_{T-h,T}(i)|^2\right] \\
&=  \sum_{i \in I_N} E\left[(\mathcal X^{N,T-h}_h(\{\kappa_i\}) + \mathcal X^{N,T-h}_h(\{\kappa_i+1\}))^2\right].  
\end{split}
\end{equation}
\end{proof}

\begin{remark}\label{rem.connection.distance.matrix}
In terms of the tree-valued model described in Section \ref{sec.tree-valued} it is also true that (compare the proof of Proposition \ref{prop.gen.distance} for the notation)
\begin{equation}
\left(\mathcal X^{N,T}_t(\{\kappa_i\}) + \mathcal X^{N,T}_t(\{\kappa_i+1\})\right)_{t \ge 0} \stackrel{d}{=} \left(\nu^{2,\mathcal U^N_{T+t}}[0,t] \right)_{t \ge 0}.
\end{equation}
\end{remark}

\begin{lemma}\label{l.stopping1}
Let 
\begin{equation}
 S^N:= \inf\left\{t \ge 0: \sum_{i = 1}^N \left(\mathcal X^{N,T}_t(\{\kappa_i\}) + \mathcal X^{N,T}_t(\{\kappa_i+1\})\right)^2 = 1\right\},
\end{equation}
then $\sup_{N \in \mathbb N} E[S^N] < \infty$. 
\end{lemma}

\begin{proof}
Observe that, by Remark \ref{rem.connection.distance.matrix} and the definition of $r$, this   is exactly the time to the most recent common ancestor, i.e. the first time where there is only one single ancestor left that gave birth to all individuals. By Proposition 6.9 in \cite{DGP12} one can bound the number of ancestors by a birth-death process with quadratic death rate.  We can now apply Theorem 3.2 and Corollary 3.4 
in \cite{Krone} to get the result. 
\end{proof}

\subsubsection{The large population limit}\label{sec.limit}

We start with the proof of Proposition \ref{lem.WFMMS}.

\begin{proof}(Proposition \ref{lem.WFMMS})
Define for $i \le  n$
\begin{equation}
b_i(\underline{x}) :=  \alpha \left(\sum_{j = 1}^{n} x_{j+n}\right)   x_i +  (\vartheta_0 x_{i+n}-\vartheta_1 x_i)
\end{equation}
and for $i \in \{n+1,\ldots,2n-1\}$
\begin{equation}
b_i(\underline{x}) :=  -\alpha \left(\sum_{j = 1}^{n} x_{j}\right) x_{i}  +  (\vartheta_1 x_{i-n}-\vartheta_0 x_{i}).
\end{equation}
then $b: \Simpl^n \to \mathbb R^{2n-1}$ is Lipschitz and satisfies 
\begin{equation}
\begin{split}
b_i(x) &\ge 0 ,\qquad \text{if} \quad x_i = 0,\\
\sum_{i = 1}^{2n-1} b_i(x) &= 0,\qquad \text{if} \quad \sum_{i = 1}^{2n-1} x_i = 1.  
\end{split}
\end{equation}
Now the result follows by Theorem 8.2.8 in \cite{EK86}. 
\end{proof}

\begin{remark}\label{r.conv.WF}
Recall the definition of  $L^N$  given in Lemma \ref{lem.WFMS.finite}. If we restrict its definition to functions $\phi \in C^2(\Simpl^n)$, where we set $x_{2n} = 1 - \sum_{i = 1}^{2n-1}x_i$ for $x \in \Simpl^n$, then it is a straight forward calculation (using Taylor expansion) to show that 
\begin{equation}
|| L^N\phi - L\phi||_{\infty} \rightarrow 0. 
\end{equation}
\end{remark}

Now, we need to define a suitable limit object for $\mathcal X^{N,T}$ given in Lemma \ref{lem.MMMS.generator}. In order to do this 
recall that we can define partial derivation  for polynomials (see \eqref{eq.polyn}) by 
\begin{equation}
\frac{\partial \Phi}{\partial \mu} (u) := \lim_{\eps\downarrow 0} \left(\Phi(\mu+\eps \delta_u) -\Phi(\mu)\right).
\end{equation}
We consider 

\begin{equation}
\Omega \Phi(\mu) := \Omega_{\text{res}} \Phi(\mu) +  \Omega_{\text{sel}} \Phi(\mu) +\Omega_{\text{mut}} \Phi(\mu),
\end{equation}
with 

\begin{equation}
\Omega_{\text{res}}\Phi(\mu) := \frac{1}{2} \int_{[0,2]}\int_{[0,2]} \frac{\partial^2 \Phi(\mu)}{\partial \mu \partial \mu} (u,v)  Q_{\mu}(du,dv),
\end{equation}

\begin{equation}
\Omega_{\text{sel}} \Phi(\mu) := \alpha \int_{[0,2]} \int_{[0,2]} \frac{\partial \Phi(\mu)}{\partial \mu} (u) \chi(v)  Q_{\mu}(du,dv),
\end{equation}
and
\begin{equation}
\Omega_{\text{mut}} \Phi(\mu) :=  \int_{[0,2]} \left( \int_{[0,2]} \frac{\partial \Phi(\mu)}{\partial \mu} (v) M(u,dv) - \frac{\partial \Phi(\mu)}{\partial \mu} (u) \right) \mu(du),
\end{equation}
where

\begin{equation}
Q_{\mu} (du,dv) = \mu(du) \delta_u (dv) - \mu(du) \mu(dv)
\end{equation} 
and 

\begin{equation}
M(u,dv) = \vartheta_1  \delta_{u+1}(dv) 1(u \le 1) +  \vartheta_0 \delta_{u-1}(dv) 1(u > 1).
\end{equation}

\begin{proposition}\label{p.martingale.FV} (Characterization and properties of the limit process)
The following holds: 
\begin{itemize}
\item[(i)] The $(\Omega,\Pi,\nu)$-martingale problem is well-posed for all initial values $\nu \in \mathcal M_1([0,2])$.
\item[(ii)] The solution has a modification with continuous paths and we denote this modification by  $\mathcal X^T$.
\item[(iii)] If $\mathcal X^{N,T}_0 \Rightarrow \mu^T$ for some random measure $\mu^T$, then 
\begin{equation}
\mathcal X^{N,T} \Rightarrow \mathcal X^T, 
\end{equation}
as processes and $\mathcal X^T_0 = \mu^T$. 
\end{itemize}
\end{proposition}

\begin{proof}
The existence of a solution follows by (iii), which is straight forward by showing convergence of the corresponding generators. 
The uniqueness can be proven either by duality or using the Girsanov transform. We refer to  \cite{dawson} or \cite{EK93} for a rigorous 
proof of (i)-(iii).
\end{proof}

\begin{lemma}\label{l.init.conv}
In our situation we have 
\begin{equation}
\mathcal X^{N,T}_0\Rightarrow \bar Y_T \lambda\big|_{[0,1]} + (1-\bar Y_T) \lambda\big|_{[1,2]},
\end{equation}
where $\bar Y_T$ is the Fisher-Wright diffusion with mutation and selection given in the introduction (note that $\bar Y_0 = E[\bar u_1]$ by the strong law of large numbers - compare Section \ref{sec.construction}) and hence the assumptions of the Proposition are satisfied. 
\end{lemma}

\begin{proof}
Note that in our case 
\begin{equation}
\mathcal X^{N,T}_0 = \frac{1}{N}\sum_{i = 1}^N  u_i^N(T) \delta_{\kappa_i} +\frac{1}{N} \sum_{i = 1}^N (1-u_i^N(T))\delta_{\kappa_i+1}
\end{equation}
and that $\frac{1}{N} \sum_{i = 1}^N u_i^N(T) = \bar Y^N_T \Rightarrow \bar Y_T$, where $\bar Y$ is the one dimensional Wright-Fisher diffusion defined in the introduction. 
Let $\phi:[0,1]\to [0,1]$ and $m \in \mathbb N$. Then 
\begin{equation}
\begin{split}
E&\left[\left(\int_{[0,2]} \phi\ d\mathcal X^{N,T}_0\right)^m\right] \\
&= \frac{1}{N^m}\sum_{i_1,\ldots,i_m} E\left[\prod_{k =1}^m \left(u_{i_k}^N(T) \phi(\kappa_{i_k}) + (1-u_{i_k}^N(T)) \phi(1+\kappa_{i_k})\right)\right]. 
\end{split}
\end{equation}
Since 
\begin{equation}
1/N^2 \sum_{i =1}^N E\left[\left|u_{i}^N(T) \phi(\kappa_{i}) + (1-u_{i_k}^N(T)) \phi(1+\kappa_{i})\right|\right] \rightarrow 0,
\end{equation}
we may assume in the above sum that the indices are pairwise different. Hence, by the independence of $(\kappa_i)$ and $(u_i)$: 
\begin{equation}
\begin{split}
E&\left[\left(\int_{[0,2]} \phi\ d\mathcal X^{N,T}_0\right)^m\right] \\
&= \frac{1}{N^m}\sum_{i_1,\ldots,i_m}  E\left[\prod_{k =1}^m \left(u_{i_k}^N(T) \phi(\kappa_{i_k}) + (1-u_{i_k}^N(T)) \phi(1+\kappa_{i_k})\right)\right]\\
&= \frac{1}{N^m}\sum_{i_1,\ldots,i_m}  E\left[\prod_{k =1}^m \left(u_{i_k}^N(T) E[\phi(\kappa_{1})] + (1-u_{i_k}^N(T)) E[\phi(1+\kappa_{1})]\right)\right] \\
&= E\left[ \left( \frac{1}{N} \sum_{i =1}^N \left(u_{i}^N(T) E[\phi(\kappa_{1})] + (1-u_{i}^N(T)) E[\phi(1+\kappa_{1})]\right)\right)^m\right] \\
&= E\left[ \left( \bar Y_T^N E[\phi(\kappa_{1})] + (1-\bar Y_T^N) E[\phi(1+\kappa_{1})]\right)^m\right] \\
&\stackrel{N\rightarrow \infty}{\rightarrow } E\left[ \left( \bar Y_T E[\phi(\kappa_{1})] + (1-\bar Y_T) E[\phi(1+\kappa_{1})]\right)^m\right].
\end{split}
\end{equation}
 Since the linear span of functions of the form $F(\mu) = (\int f d\mu)^m$ 
is an algebra that separates points, it is dense in $C(\mathcal M_1([0,2]))$ by the Stone Weierstrass  theorem and the result follows.
\end{proof}

\subsubsection{Proof of Theorem \ref{thm.key} and Lemma \ref{lem.sel.end}}\label{sec.p.key}

Recall the definition of $\mathcal E$ in Lemma \ref{lem.MMMS.generator}. Then it is not hard to see that 
\begin{equation}\label{eq.measure}
\mathcal X^{n,T} := \mathcal E(X^{n,T}),
\end{equation}
solves the $(\Omega,\Pi,\nu)$-martingale problem,  where $X^{n,T}$ is given in Proposition \ref{lem.WFMMS} with initial condition, $\nu$, 
given in Theorem \ref{thm.key}. By the well-posedness of the $(\Omega,\Pi,\mathcal X^T_0)$-martingale problem, together with the fact that analogue to Lemma \ref{l.init.conv}
$X^{n,T}_0 \Rightarrow \mathcal X^T_0$ and Lemma 4.5.1 and Remark 4.5.2 in \cite{EK86} (note that $\mathcal M_1([0,2])$ is compact), 
we get 
\begin{equation}
\mathcal X^{n,T} \Rightarrow \mathcal X^T
\end{equation}
as processes. \par 
In order to complete the proof, we need to observe that 
\begin{itemize}
\item[]{\it the above convergence as well as the convergence in Proposition \ref{p.martingale.FV} also holds when $\mathcal M_1([0,2])$ is equipped with the so called weak atomic topology.} 
\end{itemize}
We do not want to go into detail and refer to \cite{EKatomic} (Section 2 for a general 
introduction and Section 3 for the convergence of the measure-valued Moran model to the measure-valued Fleming-Viot process). The properties we need are the continuity of the map 
\begin{equation}
\mu \mapsto \int_{[0,2]} \mu(\{x\})\mu(dx)
\end{equation}
in this topology and that $\mu_n \rightarrow \mu$ in this topology for purely atomic measures $\mu,\mu_1,\mu_2,\ldots$ implies 
\begin{equation}
\sum_{i \in \mathbb N}|a_n(i) - a(i)| \rightarrow 0,
\end{equation}
where $(a(1),a(2),\ldots)$ is the reordering of the sizes of atoms of $\mu$ in a non increasing way (see Lemma 2.5 (c) in  \cite{EKatomic}). Finally observe that $\mu_n \rightarrow \mu$ with $\mu_n$ purely atomic and $\mu$ continuous implies 
\begin{equation}
\begin{split}
\int_{[0,1]} \mu_n(\{x,x+1\})\mu_n(dx) &+ \int_{(1,2]} \mu_n(\{x,x+1\})\mu_n(dx) \\
&\le 2 \int_{[0,2]} \mu_n(\{x\})\mu_n(dx) \rightarrow 0. 
\end{split}
\end{equation}
This combined with the fact that the Fleming-Viot process is purely atomic for all strict positive 
times (see Theorem 8.2.1 and Theorem 7.2.2 - compare also Section 10.1.1 - in \cite{dawson}) and Lemma \ref{lem.dist.finite} gives the result (see also Remark \ref{r.cont} below). 

\begin{remark}\label{r.cont}
(1) It is not hard to see that $\mathcal M^c([0,2])$, the space of continuous measures (i.e. measures with $\mu(\{x\}) = 0$ for all $x \in [0,2]$), is closed in the weak atomic topology. \par 
\noindent (2) In fact, what we proved above is the continuity of the map
\begin{equation}
\begin{split}
\Phi: \ &\mathcal M^c([0,2]) \cup \mathcal M^a([0,2]) \to \mathbb R_+,\\
& \mu \mapsto \int_{[0,1]} \mu(\{x,x+1\})\mu(dx) + \int_{(1,2]} \mu(\{x,x+1\})\mu(dx) 
\end{split}
\end{equation}
 in the weak atomic topology, where  $\mathcal M^a([0,2])$ is the space of purely atomic measures. Therefore, the result is a consequence of the continuous mapping theorem (see Theorem 8.4.1 in \cite{Bogachev}, Vol. II). 
\end{remark}

%

It remains to prove Lemma \ref{lem.sel.end}. Let 
\begin{equation}
\begin{split}
S^n &:= \inf\left\{t > 0:\ \exists i \text{ s.t. } X^n_t(i) + X^n_t(i+n) = 1\right\} \\
&= \inf\left\{t > 0:\ \sum_{i = 1}^n (X^n_t(i) + X^n_t(i+n))^2 = 1\right\} 
\end{split}
\end{equation}
and let $X^{n,N} = X^{2n,N}$ be the process given in Lemma \ref{lem.WFMS.finite} with initial condition $X^{n,N}_0(i) = 0$ for $i > 2n$ 
and 
\begin{equation}
\left(X^{n,N}_0(i),X^{n,N}_0(i+n)\right) = \left(\frac{\bar Y^N_T}{n},\frac{1-\bar Y^N_T}{n} \right),\qquad i = 1,\ldots,n. 
\end{equation}
 Then, in view of Remark \ref{r.conv.WF}
and Lemma 4.5.1 (see also Remark 4.5.2) in \cite{EK86} we get $X^{n,N} \Rightarrow X^n$ as processes (compare also Lemma \ref{l.init.conv} for the convergence of the initial condition). By the same argument  as in 
Lemma \ref{l.stopping1}, we get that 
\begin{equation}
\sup_{N \in \mathbb N} E[S^{n,N}] \le \sup_{N \in \mathbb N} E[S^{N}]  < \infty,
\end{equation}
where 
\begin{equation}
S^{n,N} :=\inf\left\{t > 0:\ \sum_{i = 1}^n (X^{n,N}_t(i) + X^{n,N}_t(i+n))^2 = 1\right\}.
\end{equation}

In order to see that $E[S^n] < \infty$ we can apply Skorohod's representation theorem (see Section 8.5 in \cite{Bogachev}, Vol. II) and may assume for the following that $X^{n,N} \rightarrow X^n$ almost surely. Assume now, there is a subsequence such that $S^{n,N_k} \rightarrow \bar S < \infty$. Denote by $Q : \Simpl^{n} \to [0,1], \ \underline{x} \mapsto \sum_{i= 1}^n (x_i + x_{i+n})^2$, where as always $x_{2n} = 1- \sum_{i = 1}^{2n-1} x_i$. Then, by Proposition 3.6.5 in \cite{EK86}, the continuity of $Q$ and the continuity of the process $X^n$, 
\begin{equation}
1 \equiv Q(X^{n,N_k}_{S^{n,N_k}}) \rightarrow  Q(X^n_{\bar S})
\end{equation}
and therefore $S^n \le \bar S$. Hence we get $S^n \le \liminf_{N \rightarrow \infty} S^{n,N}$ and by Fateou's lemma: 
\begin{equation}
E[S^n] \le E[\liminf_{N \rightarrow \infty} S^{n,N}] \le \liminf_{N \rightarrow \infty} E[S^{n,N}] \le \sup_{N \in \mathbb N} E[S^{n,N}] < \infty. 
\end{equation}

\newpage

\begin{appendices}

\section{Generator calculations}\label{a.gen}
Here we give the calculations needed in Section \ref{sec.proof.main}. For simplicity we will denote by $y_i := x_i$, $i = 1,\ldots,n$ the
``fit members of family $i$'' and by $z_i := x_{i+n}$, $i = 1,\ldots,n$ the ``unfit members of family $i$''. Moreover, we let 
$\bar Y:= \sum_{i = 1}^n y_i$ be the total number of fit types in the population and note that $\bar Z:= \sum_{i = 1}^n z_i = 1-\bar Y$.
Then, for example, $\phi_1(\underline{x}) = \sum_{i = 1}^n (y_i+z_i)^2$, $\phi_2(\underline{x}) = \sum_{i = 1}^n (y_i+z_i)\cdot (y_i - (y_i+z_i)\bar Y)$. \\

\begin{align*}
\bar Y^n&\sum_{i = 1}^n y_i^2 \\
&\stackrel{L^{\text{res}}}{\mapsto} \quad \bar Y^{n+1}+\left(2n+\frac{n(n-1)}{2}\right) \bar Y^{n-1}\sum_{i = 1}^n y_i^2-\left(1+2n+\frac{n(n-1)}{2}\right)\bar Y^n\sum_{i = 1}^n y_i^2,\\
&\stackrel{L^{\text{mut}}}{\mapsto} \quad 2 \vartheta_0 \bar Y^n\sum_{i = 1}^n y_i z_i
+n\vartheta_0 \bar Y^{n-1}\sum_{i = 1}^n y_i^2 - (n\vartheta_0+(n+2)\vartheta_1) \bar Y^n\sum_{i = 1}^n y_i^2, \\
&\stackrel{L^{\text{sel}}}{\mapsto} \quad (n+2) \alpha \bar Y^n\sum_{i = 1}^n y_i^2- (n+2) \alpha \bar Y^{n+1}\sum_{i = 1}^n y_i^2,
\end{align*}

\begin{align*}
\bar Z^{n} &\sum_{i = 1}^n z_i^2 \\
&\stackrel{L^{\text{res}}}{\mapsto}\quad \bar Z^{n+1}+ \left(2n+\frac{n(n-1)}{2}\right) \bar Z^{n-1} \sum_{i = 1}^n z_i^2-\left(1+2n+\frac{n(n-1)}{2}\right) \bar Z^{n} \sum_{i = 1}^n z_i^2,\\
&\stackrel{L^{\text{mut}}}{\mapsto}\quad 2 \vartheta_1 \bar Z^n\sum_{i = 1}^n y_i z_i +n\vartheta_1 \bar Z^{n-1} \sum_{i = 1}^n z_i^2 -  (n\vartheta_1 + (n+2)\vartheta_0) \bar Z^{n} \sum_{i = 1}^n z_i^2,\\
&\stackrel{L^{\text{sel}}}{\mapsto}\quad -(n+2) \alpha \bar Z^{n} \sum_{i = 1}^n z_i^2  + (n+2) \alpha\bar Z^{n+1} \sum_{i = 1}^n z_i^2 ,
\end{align*}

\begin{align*}
\bar Y^{n} &\sum_{i = 1}^n y_i z_i\\
&\stackrel{L^{\text{res}}}{\mapsto}\quad  \left(n+\frac{n(n-1)}{2}\right)\bar Y^{n-1} \sum_{i = 1}^n y_i z_i - \left(1+2n+\frac{n(n-1)}{2}\right)\bar Y^{n} \sum_{i = 1}^n y_i z_i ,\\
&\stackrel{L^{\text{mut}}}{\mapsto}\quad   \vartheta_1 \bar Y^{n} \sum_{i = 1}^n y_i^2 + \vartheta_0\bar Y^{n} \sum_{i = 1}^n z_i^2  + n\vartheta_0\bar Y^{n-1} \sum_{i = 1}^n y_i z_i -(n+1)(\vartheta_0+\vartheta_1) \bar Y^{n} \sum_{i = 1}^n y_i z_i ,\\
&\stackrel{L^{\text{sel}}}{\mapsto}\quad  (n+1)\alpha \bar Y^{n} \sum_{i = 1}^n y_i z_i  - (n+2)\alpha\bar Y^{n} \sum_{i = 1}^n y_i z_i.
\end{align*}

\end{appendices}

\newpage


\end{document}